\documentclass{amsart}
\usepackage{verbatim,amssymb,amsmath,amscd,latexsym,amsbsy,mathrsfs} 
\usepackage{graphicx}
\usepackage[all]{xy}
 \input{xy}
 \xyoption{all}

\newtheorem{thm}{Theorem}[section]
\newtheorem{cor}[thm]{Corollary}

\newtheorem{prop}[thm]{Proposition}

\newtheorem{lemma}[thm]{Lemma}
\newtheorem{rmk}[thm]{Remark}

\newtheorem{con}[thm]{Conjecture}
\newtheorem{ex}[thm]{Example}

\DeclareMathOperator*{\im}{im} 
\DeclareMathOperator*{\coker}{coker} 
 
\DeclareMathOperator*{\Spec}{Spec}

\newcommand{\PP}{\mathbb{P}}

\newcommand{\cC}{\mathcal{C}}

\newcommand{\cX}{\mathcal{X}}

\newcommand{\D}{\mathcal{D}}
\newcommand{\cL}{\mathcal{L}}

\newcommand {\C} {{\mathbb C}}

\newcommand {\Z} {{\mathbb Z}}
\newcommand {\Q} {{\mathbb Q}}

\newcommand {\F} {{\mathcal F}}

\newcommand {\E} {{\mathcal E}}
\newcommand {\dt} {{\bullet}}

\newcommand {\OO} {{\mathcal O}}
\newcommand {\A} {\mathbb{A}}
\newcommand {\I} {\mathcal{I}}

\begin{document}
\title{A Lefschetz $(1,1)$ theorem for singular varieties}
\author{
        Donu Arapura    
}
 \thanks {Partially supported by the NSF }
\address{Department of Mathematics\\
 Purdue University\\
 West Lafayette, IN 47907\\
U.S.A.}
 \maketitle

\tableofcontents

The usual Lefschetz $(1,1)$ theorem says that given a smooth complex
projective variety $X$, an element of $H^2(X,\Z)$ is the  class of a divisor if and only
it lies $H^{11}$ or equivalently in $F^1H^2(X,\C)$. When
$X$ is singular, $H^2(X,\C)$ still carries the Hodge filtration associated
to the canonical  mixed Hodge structure. One of our main results is that
an element of $H^2(X, \Z)$ lies in $F^1$ if and only if it comes from motivic
cohomology $H^2_M(X,\Z(1))$. Along the way, we give a reasonably
concrete description of the last group. As in Deligne's original
construction of mixed Hodge structures, one starts by building a
suitable simplicial resolution 
$$
\xymatrix{
 \ldots & \tilde X_2\ar[r]\ar@<1ex>[r] \ar@<-1ex>[r]& \tilde X_1\ar[r]_{p_1}\ar@<1ex>[r]^{p_0} & \tilde X_0\ar[r]^{\pi} & X
}
$$
Very loosely, $\tilde X_\dt$ is a diagram of smooth varieties with the
same cohomology as $X$. An element of $H^2_M(X,\Z(1))$ 
is represented by  a pair $(D,f)$, where $D$ is a divisor on $\tilde
X_0$ and $f$ a rational function on $\tilde X_1$,  such that $\partial
D:=p_0^*D-p_1^*D$ is defined and equal to  $(f)$ and $\partial f=1$. 
An example of such a pair is $(\pi^*C, 1)$
where $C$ is a Cartier divisor on $X$. So in this sense, the elements
of $H^2_M(X,\Z(1))$ can be viewed as generalized Cartier divisors on
$X$. It is worth noting that Barbieri-Viale and Srinivas \cite{bs} have
constructed a normal projective  surface where not every element of 
$H^2(X,\Z)\cap F^1$ can be represented by a Cartier divisor, so  generalized divisors
are really needed here. 

In addition to the Lefschetz theorem,  one of our goals is to
give a conjectural description of weight $2p$ Hodge cycles on
$H^{2p}(X,\Q)$, or equivalently elements of $H^{2p}(X,\Q)\cap F^p$, for
all degrees $p$.
As  a first step, we will try to  understand what happens on the maximal  pure quotient
$\tilde H^{2p}(X):=H^{2p}(X)/W_{2p-1}$. We define a class in $\tilde
H^{2p}(X)$ to be homologically Cartier if it is represented by an
algebraic cycle on some resolution. (Since this notion seems more
 broadly useful, we modify this  definition to work  in arbitrary
 characteristic in the first section.)
 Basic examples of homological
Cartier cycles are provided by Chern classes of vector bundles, Weil
divisors on normal surfaces and more generally numerically Cartier $\Q$-divisors in the sense
 of Boucksom, de Fernex, Favre  and Urbanati \cite{bff}.
The  Hodge conjecture would imply that any Hodge cycle on $\tilde
H^{2p}(X)$ is given by a homologically Cartier cycle. This however only
gives a partial solution to the original problem, since  there
is in general a nontrivial obstruction $\varepsilon(\alpha)$ for a
homologically Cartier cycle $\alpha$ to lift to
a Hodge cycle on $H^{2p}(X)$. The search for a natural source of 
unobstructed classes led the author first to operational Chow groups
and then to  motivic cohomology. By work of Kimura \cite{kimura}, the
operational Chow group $CH_{OP}^p(X)$ can be identified with the
kernel $\ker \partial:CH^p(\tilde
X_0)\to CH^p(\tilde X_1)$, for suitable $\tilde X_\dt\to X$. It follows
that any element $\alpha\in CH^p_{OP}(X)_\Q$   will give a
homologically Cartier cycle, but $\varepsilon(\alpha)$ might still be nonzero. What is required
is a further constraint on the cycle $\partial \alpha$, and this is where
motivic cohomology enters the picture.

 For our purposes, the most congenial
approach to motivic cohomology is due to  Hanamura
\cite{hanamura}. He defines it as  the cohomology of a double complex built from
Bloch's cycle complex and  a  simplicial resolution. 
Hanamura shows that, with $\Q$-coefficients, the result is well
defined and functorial.
However, we really need this with $\Z$-coefficients. We handle this by
showing that the  group defined using Hanamura's approach
 coincides with  the more intrinsic definition given by Friedlander,
Suslin and Voevodsky in their book (specifically \cite{fv}) as the cohomology of a complex of
sheaves on the cdh site. Although these results are  probably known to
some, we include proofs
in sections 4 and 5 for lack of a suitable reference.
 Returning to the previous discussion, we have a map from motivic
 cohomology $H^{2p}_M(X,\Q(p))\to CH^p_{OP}(X)_\Q$, if 
$\alpha\in CH^p_{OP}(X)_\Q$ lifts we show that
$\varepsilon(\alpha)=0$, so in particular it determines a  weight $2p$ Hodge cycle on $H^{2p}(X,\Q)$. The
proof uses
explicit formulas for higher cycle classes due to Kerr, Lewis and
M\"uller-Stach \cite{klm}. 
This result leads naturally
to a refined Hodge conjecture (conjecture~\ref{con:hodge}) that if $X$ is defined over $\overline{\Q}$,
then any weight $2p$  Hodge cycle on $H^{2p}(X,\Q)$ comes from motivic
cohomology.  Unlike the usual Hodge conjecture, the statement
is easy to  falsify  in general  for  varieties not 
defined over $\overline{\Q}$. This is closely related to the fact that  kernels of
Abel-Jacobi maps on Chow groups of  transcendental varieties can be
very large. By contrast, according to a conjecture of Bloch and
Beilinson, this sort of phenomenon should not occur for varieties over $\overline{\Q}$.
Regarding evidence for conjecture~\ref{con:hodge},  we note that
it  holds for $p=1$  by the   Lefschetz theorem stated above and proved in
 section 7. As a consequence it also holds for products of degree 2
 Hodge cycles.
In the last section, we prove that the conjecture holds for the $n$-fold self fibre
product of an elliptic modular surface. The result is
deduced by showing that the algebra of  Hodge cycles on these
varieties are generated by
degree 2 Hodge cycles following a careful analysis of the Leray
spectral sequence.

The word ``variety'' will mean a reduced scheme of finite type over the
ground field, which, with 
the exception of the first section, is always $\C$. We write $H^*(X)$
(respectively $H^*(X,\Z)$)
for singular cohomology of the associated analytic space with
coefficients  in $\Q$ (respectively $\Z$) in all but the first section.

Comments by V. Srinivas, B. Totaro and A. Vistoli at an early stage of this project were
very helpful in steering me in the right direction. Parts of this paper were written during a
short but productive visit to the  Simons Center in Stony Brook.

\section{Homologically Cartier cycles}

In this section we work over an arbitrary algebraically closed field $k$,
but over $\C$ in the remaining sections. Let  $H^*(-)$ denote either
$\ell$-adic cohomology, with $\Q_\ell$ -coefficients, where
$\ell\not=char\, k$, or singular cohomology with $\Q$-coefficients
when $k=\C$. Let $H_*(-)$ denote either or ordinary or $\ell$-adic
Borel-Moore homology \cite{fulton, laumon}, again with $\Q$ or
$\Q_\ell$ coefficients. Every $p$-dimensional closed subvariety
$V\subset X$ possesses a fundamental class $[V]\in H_{2p}(X)$.
Let $C_p(X)\subset H_{2p}(X)$ denote the $\Q$-span of these classes.
This can be identified with the quotient of the Chow group $CH_p(X)$
tensored with $\Q$ by homological equivalence. 
At this point, we need to bring the weight filtration into play. We
start with some elementary definitions and properties.

\begin{lemma}\label{lemma:W}
  Let $\pi:\tilde X\to X$ be a nonsingular alteration of a projective
  variety $X$. The  subspaces
$$W_{p-1} H^p(X)= \ker[H^p(X)\to H^p(\tilde X)]$$
$$W_{-p}H_{p}(X) = \im [H_p(\tilde X)\to H_p( X)]$$
are independent of the choice of $\tilde X$.
\end{lemma}

\begin{proof}
  Given a second alteration $\pi':\tilde X'\to X$, after replacing it by 
  the component of an alteration of $\tilde X\times_X \tilde X'$
  dominating $X$, we can assume that $\tilde X'$ factors through
  a morphism $\tilde X'\to \tilde X$. By Poincar\'e duality,
  $H^p(\tilde X)\to H^p(\tilde X')$ is injective. Therefore $\ker
  \pi^*=\ker \pi'^*$. The second part is similar.
\end{proof} 

We can see easily that $\bigoplus W_{*-1}H^*(X)\subset H^*(X)$ is an
ideal. Set
$$\tilde H^*(X) = H^*(X)/\bigoplus W_{*-1}H^*(X) \cong \im [H^*(X)\to
H^*(\tilde X)]
$$
to the quotient ring. Also let
$$\tilde H_j(X)= W_{-j}H_j(X)$$

\begin{lemma}
  If $\alpha\in W_{i-1}H^i(X)$ and $\beta\in \tilde H_j(X)$, then $\alpha\cap\beta=0$.
\end{lemma}

\begin{proof}
  Choose $\tilde \beta\in H_j(\tilde X)$, with $\tilde X$ as above, so
  that $\pi_*\tilde \beta=\beta$. Then 
$$\alpha\cap \beta =\pi_*(\pi^*\alpha\cap \tilde \beta) =0$$
\end{proof}

It follows that the cap product  descends to  a well defined pairing
$$\tilde H^i(X)\otimes \tilde H_j(X)\to \tilde H_{j-i}(X)$$
that we will also refer to as a cap product.

\begin{lemma}
  The image of the cycle map $CH_p(X)\to H_{2p}(X)$ lies in $\tilde H_{2p}(X)$
\end{lemma}

\begin{proof}
  Given an algebraic cycle $\beta\in CH_p(X)$, we can find an algebraic
  cycle $\tilde \beta\in CH_*(\tilde X)$ such that $\pi_*\tilde \beta=\beta$.
\end{proof}

We come to the key definition. If $X$ is a projective (possibly
reducible) variety, an element $\alpha\in \tilde H^{2p}(X)$ can be regarded
as an element of $H^{2p}(\tilde X)$, for any alteration, under the
inclusion  $\tilde H^{2p}(X)\subset H^{2p}(\tilde X)$. We say that
$\alpha$ is
{\em homologically Cartier} if  it is represented by an algebraic
cycle on some nonsingular alteration $\tilde X$. 
 Let $C^p(X)$ denote the space
homologically Cartier cycles on $X$. When $X$ is nonsingular and
 $\pi^*\alpha$ is algebraic then so is $\alpha=\pi_*\pi^*\alpha$. Thus
 we see that homologically Cartier cycles are just algebraic cycles in
 this case. For similar reasons, we can see that if $X$ is irreducible
 of dimension $n$, then
$$C^p(X)= \im H^{2p}(X)\cap C_{\dim X-n}(\tilde X)$$
for a fixed   nonsingular alteration $\tilde X\to X$

\begin{prop}
\-
  \begin{enumerate}
  \item $C^p(-)$ is functorial in the sense that if $f:X\to Y$ is morphism, then
    $f^*C^p(Y)\subset C^p(X)$.
\item $C^*(X)\subset \tilde H^{2*}(X)$ is a subring.
\item If  $\alpha\in C^p(X)$ and $\beta\in C_{q}(X)$, then $\alpha\cap \beta\in C_{q-p}(X)$
  \end{enumerate}
\end{prop}

\begin{proof}
  The first property is clear, because we can find a commutative
  diagram
$$
\xymatrix{
 \tilde X\ar[r]\ar[d] & \tilde Y\ar[d] \\ 
 X\ar[r] & Y
}
$$
where the vertical maps are nonsingular alterations.
For (2), it is enough to observe that $C^*(X)$ is an intersection of
two subrings of $H^{2*}(\tilde X)$ namely $\im H^{2*}(X) \cap
C^*(\tilde X)$. To prove (3), choose $\tilde \alpha\in
CH^p(\tilde X)=CH_{\dim X-p}(\tilde X)$ with $[\tilde \alpha]=\alpha$ and
$\tilde \beta\in
CH_{q}(\tilde X)$ with $\pi_*[\tilde \beta] =\beta$ (the existence of
 $\tilde \beta$ is easy c.f. \cite[prop 1.3]{kimura}). Then we have
$\alpha\cap \beta= [\pi_*(\tilde \alpha\cdot \tilde \beta)]$.
\end{proof}

\begin{prop}\label{prop:HX}
  Given a projective variety $X$, choose a nonsingular alteration
  $\pi:\tilde X\to X$ and a nonsingular alteration $\tilde X_1\to \tilde
X_0\times_{X}\tilde X_0$ with projections $p_i:\tilde X_1\to \tilde
X_0$. Then
$$ H^i(X)\to  H^i(\tilde X)\stackrel{\partial}{\to} H^i(\tilde X_1)$$
is exact, where $\partial=p_1^*-p_2^*$.
\end{prop}

\begin{proof}
Suppose  that $char\, k=0$. Since the \'etale cohomology of $X$ is invariant under base extension
to a larger algebraically closed field, there is no loss assuming that $k=\C$. By 
the comparison theorem, we can also assume that 
  $H^*(X)$ is singular cohomology. Now the proposition follows from
  \cite[prop 8.2.5]{deligneH}.

When $char\, k= r>0$, we  also use a weight argument, but we will need to
work out things from scratch. First of all, we can reduce to
  the case where $k$ is the algebraic closure of a field $k_0$ which is
  finitely generated over the finite field $\mathbb{F}_{r}$. We can assume that
  $X$ is defined by the base change of a variety defined over $k_0$. Using \cite{dJ}, 
we build a smooth simplicial scheme $\tilde X_\dt\to X$ augmented over
$X$ as follows. Let $\tilde X_0= \tilde X$ and $\tilde X_1$ as
above. Choose the higher  $\tilde X_n$  inductively so that the canonical map
$$\tilde X_n\to cosk(sk_{n-1} \tilde X_\dt)_{n}$$
is proper and surjective; see \cite[\S 6]{deligneH} or \cite{s}.
This will ensure that $\tilde X_\dt\to X$ will satisfy cohomological
descent, and in particular that we have a descent spectral sequence
$$E_1^{pq}= H^q(\tilde X_p) \Rightarrow H^{p+q}(X)$$
We can assume that for any fixed constant $N$, all $\tilde X_p$ for $p\le N$, and maps between them,
are defined over  $k_0$, after possibly
enlarging it. This will ensure that  $G=Gal(k/k_0)$ will act
on the spectral sequence in the range $p\le N$. In particular, that
the differentials are equivariant in this range.
Choose $\phi\in G$ which maps to a Frobenius in $Gal(\overline{\mathbb{F}_{r}}/\mathbb{F}_{r^s})$.
The Weil conjectures \cite{deligneW} will show that the eigenvalues of
$\phi$ on $E_1^{pq}$ and $E_1^{p'q'}$ are different whenever
$q\not= q'$. Since $N$ can be chosen arbitrarily large,
this forces degeneration of the spectral
sequence at $E_2$. In particular,  $E_2^{0i}=E_\infty^{0i}$. This
implies exactness of
$$H^i(X)\to H^i(\tilde X_0)\to H^i(\tilde X_1)$$
\end{proof}

\begin{cor}\label{cor:Cseq}
  We have an exact sequence
$$0\to C^p(X)\to C^p(\tilde X_0)\to C^p(\tilde X_1)$$
\end{cor}

\begin{proof}
By definition the first arrow is injective. The sequence is also clearly
a complex by functoriality of $C^p(-)$ We just have to show that if $\alpha\in
C^p(\tilde X_0)$ maps to $0$ in the third group, then it must come
from the first.
 We  have a  commutative diagram
$$
\xymatrix{
 0\ar[r] & C^p(X)\ar[r]\ar@{-->}[d] & C^p(\tilde X_0)\ar[r]\ar[d] & C^p(\tilde X_1)\ar[d] \\ 
 0\ar[r] & \tilde H^{2p}(X)\ar[r] & H^{2p}(\tilde X)\ar[r] & H^{2p}(\tilde X_1)
}
$$
We see that $\alpha$ maps to $0$ in
$H^{2p}(\tilde X_1)$. Therefore it lies in $C^p(X)=\tilde H^{2p}(X)\cap
C^{p}(\tilde X_0)$ by exactness of the bottom row.
\end{proof}

\begin{rmk}\label{rmk:1}
  It will be useful to say a few words about the geometry of $\tilde
  X_1$, when $\tilde X\to X$ is desingularization and $X$ is
  irreducible. Then we may also assume that $\tilde X$ is
  irreducible.  Let $\Sigma\subset X$ the maximal closed set over
  which $f$ is not an isomorphism. If $E= f^{-1}\Sigma$, then
$\tilde X\times_X\tilde X$ is a union of $\tilde X$ embedded
diagonally and $E\times_\Sigma E$. Thus $\tilde X_1$ can be taken to
be a disjoint union of $\tilde X$ and an alteration $\tilde X_1'$  of
$E\times_\Sigma E$. The proposition and its corollary holds when
$\tilde X_1$ is replaced by $\tilde X_1'$.
\end{rmk}

Let us discuss some examples. Suppose that  $E$ is a vector bundle on $X$. The usual cohomological 
Chern class $c_p(E)\in H^{2p}(X)$ is homologically Cartier because it pulls back to an algebraic cycle on $\tilde X$.
Let us say that a cycle is {\em Cartier } if it is $\Q$-linear combination of Chern classes of vector bundles.
We will see later that not every homologically Cartier cycle is Cartier.  We   lay the groundwork now, by
giving a different source of examples.
Suppose that $X$ is a normal projective surface with  a desingularization $\pi:\tilde X\to X$ with exceptional divisors
$E_i$. Given a Weil divisor $D$ on $X$, with strict transform $D'$,  Mumford \cite{mumford}  constructed 
a unique $\Q$-divisor $\pi^*D = D'+\sum a_i E_i$ on $\tilde X$ for which $\pi^*D\cdot E_j=0$ for all $j$. We have a Mayer-Vietoris  type
sequence
$$H^2(X)\to H^2(\tilde X)\to \bigoplus H^2(E_j)$$
which shows  that $[\pi^*D]\in H^2(X)$. Thus we have proved:

\begin{lemma}\label{lemma:mumford}
A Weil divisor $D$ on a normal projective surface gives a homologically Cartier cycle, namely $\pi^*D$.
\end{lemma}

 If $X$ is a higher dimensional normal projective variety
over a field of characterstic zero,  Boucksom, de Fernex, Favre  and Urbanati \cite{bff} generalize this as follows.
They call a Weil divisor $D$ on $X$ numerically $\Q$-Cartier if on some desingularization $\pi:\tilde X\to X$, there exists
a (necessarily unique) $\Q$ divisor $\pi^*D$ on $\tilde X$ which is $\pi$-trivial and for which $\pi_*\pi^*D=D$. The $\pi$-triviallity condition
means that the intersection number $\pi^*D\cdot C=0$ for any curve that gets contracted under $\pi$.  We claim that $[\pi^*D]\in \im H^2(X)$.
This will imply that $\pi^* D$ is homologically Cartier; in fact, the conditions of being  homologically Caritier is really the same as the condition
of being  numerically Cartier  in this case. The claim follows from the  next lemma.

\begin{lemma}
 A  $\Q$-divisor $F$ on $\tilde X$ is $\pi$-trivial if and only if $[F]\in \im H^2(X)$.
\end{lemma}

\begin{proof}

One direction is clear, if $[F]\in \im H^2(X)$, then $F\cdot C = \pi_*[F]\cdot 0=0$ for any curve $C$ contracted by $\pi$.

The converse hinges on the well known fact that numerical equivalence and homological equivalence for divisors
on a smooth projective variety coincide (because the Neron-Severi group tensor $\Q$ injects into $H^2$).
Thus using proposition \ref{prop:HX}, we have to show that $p_1^*F-p_2^*F$ is numerically trivial
on $\tilde X_1$, where $p_i:\tilde X_1\to \tilde X$ are the projections. Let $C\subset \tilde X_1$ be an irreducible curve.
Let $C_i=p_i(C)$ and $C'= \pi\circ p_1(C)$ with reduced structures.  The diagram
$$
\xymatrix{
 C\ar[r]^{d_1}\ar[d]^{d_2} & C_1\ar[d]^{e_1} \\ 
 C_2\ar[r]^{e_2} & C'
}
$$
commutes. First, let us suppose that $C'$ is a curve. Let us replace the curves in the diagram by their normalizations.
The degrees  of the maps  $d_i, e_i$ are indicated in the diagram.
We have  that $d_1e_1=d_2e_2$ is the degree of $C\to C'$.  
Let $F'$ be the pushforward of the zero cycle $F|_C$ under $C\to C'$. Then
$$(p_1^*F-p_2^*F)\cdot C= d_1F\cdot C_1 - d_2 F\cdot C_2= d_1e_1\deg(F') - d_2e_2\deg(F')=0$$
The remaining case is when $C'$ is a point. Then $C_i$ is either a point or  a curve contracted by $\pi$. In
either case $p_i^*F\cdot C = F\cdot p_{i*}C = 0$.

\end{proof}

There are various ways in which this construction extends to higher rank
sheaves. We look at a particularly simple case.
Suppose that $X$ is a smooth projective
variety over a field of characteristic $0$ on which a finite group $G$
acts. Then the quotient $Y=X/G$ is well known to exists in the category of normal
projective varieties. Let $E$ be reflexive sheaf on $Y$.  Then we can
homologically Cartier ``Chern classes'' $c_p(E)\in H^{2p}(Y)$. Here
is the construction: $E$
restricts to a locally free sheaf on the smooth locus $U$. This can be
pulled  back to the preimage of $U$ in $X$ and extended to give a
locally free sheaf $F$ on $X$. The Chern classes $c_p(F)\in H^{2p}(X)$
are necessarily $G$-invariant, so they define cohomology classes on
$Y$ thanks to the isomorphism $H^{2p}(Y)= H^{2p}(X)^G$. These are
homologically Cartier by definition because  $X\to Y$ is a nonsingular alteration.

An additional source of examples of homologically Cartier cycles will be discussed in section \ref{sect:opchow}.

\section{Hodge cycles}\label{sect:HC}

In this section, we work exclusively over $\C$ and take $H^*(X)$ to be singular cohomology with its canonical mixed
Hodge structure \cite{deligneH}.  The quotient $\tilde H^i(X) = H^i(X)/W_i$ is a pure Hodge structure of weight $i$.
By a Hodge cycle of weight $2p$ on  a mixed Hodge structure $H$,  we will mean an element of
$$Hom_{MHS}(\Q(-p), H)\cong Hom_{MHS}(\Q(0),H(p)).$$
 More concretely, this is given an element of
$(2\pi i)^pH_\Q\cap W_{2p}\cap F^pH$, or simply $(2\pi i)^pH_\Q \cap
F^pH$ when $W_{2p}=H$. Let us now normalize things so that when $X$ is smooth and projective,  the
image of the cycle map on $CH^p(X)_\Q$  lies in $H^{2p}(X, \Q(p))=H^{2p}(X,(2\pi i)^p\Q)$ (as a lattice  $H^{2p}(X, \C)$). 
This will make certain statements appear more natural.
Here is the key observation:

\begin{prop}
 If $X$ is  a projective variety, 
  the image of $C^p(X)\to \tilde H^{2p}(X, \Q(p))$ consists of Hodge cycles
  of weight $2p$. The converse is  true if   the Hodge conjecture, in
  degree $2p$,   holds for a resolution of $X$. In particular, the
  result holds unconditionally for $p=1$.
\end{prop}

\begin{proof}
 This follows immediately from the diagram given  in the proof of
 corollary \ref{cor:Cseq}.
\end{proof}

From the extension
$$0\to W_{2p-1}H^{2p}(X)\to H^{2p}(X)\to \tilde H^{2p}(X)\to 0$$
together with the fact that
$$Hom_{MHS}(\Q(-p), W_{2p-1}H^{2p}(X))=0$$
we obtain an injective map:

\begin{lemma}\label{lemma:HCinj}
  $$ Hom_{MHS}(\Q(-p), H^{2p}(X))\hookrightarrow Hom_{MHS}(\Q(-p), \tilde
H^{2p}(X))$$
\end{lemma}

A homologically Cartier cycle $\alpha$ gives an element
$Hom_{MHS}(\Q(-p), \tilde H^{2p}(X))$
Under the connecting map, we obtain a class
$$\varepsilon(\alpha)\in Ext_{MHS}^1(\Q(-p), W_{2p-1}H^{2p}(X))$$
Let
$$\varepsilon_1(\alpha)\in Ext_{MHS}^1(\Q(-p), Gr^W_{2p-1}H^{2p}(X))$$
denote the image of the previous class in the $Ext$ group above. These
give the obstructions to lifting $\alpha$ to a Hodge cycle  in
$H^{2p}(X)$ and $H^{2p}(X)/W_{2p-2}$ respectively. 

We want to describe
$\varepsilon$ and $\varepsilon_1$ in more explicit terms.
By  work of Carlson \cite{carlson},  the above two $Ext$ groups can be
identified with the intermediate Jacobians
\begin{equation}
  \label{eq:epsilon}
JW_{2p-1}H^{2p}(X) =  \frac{W_{2p-1}H^{2p}(X)}{F^pW_{2p-1}H^{2p}(X,\C)+ W_{2p-1}H^{2p}(X,\Q)}  
\end{equation}
and 
\begin{equation}
  \label{eq:epsilon1}
 JGr^W_{2p-1}H^{2p}(X)=
\frac{Gr^W_{2p-1}H^{2p}(X)}{F^pGr^W_{2p-1}H^{2p}(X,\C)+
  Gr^W_{2p-1}H^{2p}(X,\Q)}
\end{equation}
respectively.
Carlson  gives a recipe for computing the  extension classes. Choose lifts  (which exist) $A\in F^pH^{2p}(X)$ and
$B\in H^{2p}(X,\Q)$ of the Hodge cycle $[\alpha]$, then
the difference $A-B$ lies in $W_{2p}H^{2p}(X)$.
The obstruction
$\varepsilon(\alpha)$
is the class of $A-B$ in the quotient in \eqref{eq:epsilon}. We can also consider a sequence
of intermediate obstructions $\varepsilon_1(\alpha),\ldots$ given by
the projection of $\varepsilon(\alpha)$ to  \eqref{eq:epsilon1} etc.
To proceed further,
fix a smooth projective augmented simplicial scheme $\tilde X_\dt\to
X$ satisfying cohomological descent. We only require that this be a
semi or strict simplicial object, which means that there are face maps
$p_i:\tilde X_j\to \tilde X_{j-1}$, but not degeneracy maps in the
backwards direction. In practice, this makes the constructions more
economical (see proposition~\ref{prop:gnpp}). Let $(\E^\dt(\tilde X_j),d)$ denote the de Rham
complex. This forms a double complex $(\E^\dt(\tilde X_\dt), d,\pm \partial)$,
where $\partial=\sum (-1)^ip_i^*$ denotes the simplicial boundary  (we work up to sign). We can form the total complex,
$E^n= \bigoplus_{a+b=n} \E^a(\tilde X_b)$, with differential
$d\pm \partial$.  This is filtered by
$F^p E^\dt = \bigoplus_{a\ge p} \E^{a,b}(\tilde X_c)$. Then $A$ can be
represented by an element  
$$A=(A_0,A_1,\ldots) \in F^pE^{2p} = F^p\E^{2p}(\tilde X_0)\oplus
F^p\E^{2p-1}(\tilde X_1)\oplus\ldots$$
 Similarly $B$ can be
represented by an element  $(B_0,\ldots) \in E_\Q^{2p}$, where $E_\Q$ denotes the total
complex of the $C^\infty$ singular cochain complex with coefficients
in $\Q$. To make sense of $A-B$, we can either push $A$ into
$E_\Q^{2p}\otimes \C$ under the  quasi-isomorphism $\E^\dt\to \E^\dt_\Q\otimes
\C$ defined by integration; or we can replace $(B_0,\ldots)$ by a
sequence of differential forms with rational periods. In the second
case, we may assume that $A_0=B_0$. We are now in a position to
extract an explicit description. The torus  $JGr^W_{2p-1}H^{2p}(X)$ is
a subquotient of  the Griffiths' intermediate Jacobian
$$J^p(\tilde X_1)_\Q= \frac{H^{2p-1}(\tilde X_1)}{F^p H^{2p-1}(\tilde X_1)+
  H^{2p-1}(\tilde X_1,\Q)}$$
Any homologically trivial cycle $Z$ on $\tilde X_1$ determines an
element $AJ(Z)\in J^p(\tilde X_1)_\Q$. We will recall the construction in
the proof below.

\begin{prop}\label{prop:epsilon1}
  If $\alpha$ is a homologically Cartier cycle, $\varepsilon_1(\alpha)$ is
(up to sign)  the image 
$ AJ(\partial\alpha)$ under the map $\ker[ J^p(\tilde X_1)\to J^p(\tilde X_2)]_\Q\to JGr^W_{2p-1}H^{2p}(X)$.
\end{prop}

\begin{proof}
The expressions $\varepsilon_1(\alpha)$ and
$ AJ(\partial\alpha)$ will be summed over the connected components of
$\tilde X_1$. So without loss of generality, we can assume that it is
connected of dimension $n$. 
We have that $\partial \alpha$ is
homologically trivial. Therefore it is the
  boundary of a rational $C^\infty$ ($2n-2p+1$)-chain $\Gamma$.
Under Poincar\'e duality,
\begin{equation}
  \label{eq:PoincF}
F^{n-p+1}H^{2n-2p+1}(\tilde X_1)^*\cong H^{2p-1}(\tilde X_1)/F^p
H^{2p-1}(\tilde X_1)   
\end{equation}
Integration along $\Gamma$ defines a functional on $H^{2n-2p+1}(\tilde
X_1)$,
and therefore an element of the right side of \eqref{eq:PoincF}.
Its image  in $J^p(\tilde X_1)_\Q$ is precisely $AJ(\partial\alpha)$.

We assume that $B_i$ is a sequence of differential forms with
rational periods, and that $A_0=B_0$. Then $A_1-B_1$ determines a
closed form whose image in $JGr^W_{2p-1}H^{2p}(\tilde X_1)$ is
$\varepsilon_1(\alpha)$. Regarding $B_1$ as a current, we can choose it cohomologous to
the current $\gamma$ given by
$$\omega\mapsto \int_\Gamma\omega$$
The form $A_1$ defines the current
$$\omega\mapsto \int_{\tilde X_1}A_1\wedge\omega$$
which acts trivially on the left side of \eqref{eq:PoincF}. Therefore the
action of  $\pm(A_1-B_1)$ on  \eqref{eq:PoincF}  is
integration on $\Gamma$.

\end{proof}

We now give a simple example, where this obstruction is nontrivial.

\begin{ex}\label{ex:1}
  Let $C\subset\PP^2_\C$ be a nonsingular cubic. Let $Q_0\subset
  \PP^2$ be a very general quartic. The two curves meet in 12 very
  general points,
  $p_1,\ldots, p_{12}$. Blow up these points to get a surface $f:\tilde
  X\to \PP^2$ with exceptional divisors $E_1,\ldots, E_{12}$. 
Let $\tilde C\subset \tilde X$ be the strict transform of $C$ which is
abstractly the same curve. Let
$Q= f^*Q_0-\sum E_i$. We have that $Q^2=4$ and $Q\cdot \tilde
C=0$. Furthermore, $|Q|$ is base point free, so it contracts $\tilde
C$ to a point $p$ in a normal surface $X$. We build the augmented simplicial scheme 
$$\tilde X_1= \tilde C\rightrightarrows \tilde X_0= \tilde X\coprod
p\to X$$
Since $\tilde X_2=\emptyset$ and $H^1(\tilde X_0)=0$, we have $J(C)= JGr^W_{1}H^{2}(X)$.
Let $D=f^*L-E_1-E_2-E_3$, where
$L\subset \PP^2$ is a line. Then $D$ has degree $0$ on $\tilde C$, so
it is gives a homologically Cartier cycle on $X$.  Note however the
class of $D$ in the Jacobian of $\tilde C$ is nonzero because, the points
$p_1,p_2,p_3$ were very general and therefore noncolinear. So $\varepsilon_1(D)\not=0$.
\end{ex}

We want to say more about $\varepsilon(\alpha)$ when $p=1$ and $X$ is eventually
a surface.  In this
case, we will work integrally.  We define $W_1H^1(X,\Z)$ to be the 
intersection $W_1 H^2(X,\Q)$ with the torsion free part of $H^2(X,\Z)$.
Choose a simplicial scheme $\tilde X_\dt\to
X$  as above.
Let $Div_g(\tilde X_1)$ denote the space of divisors in general position with
respect to the maps $p_i$; more precisely, no component of $D\in
Div_g(\tilde X_1)$ should contain
the image of a component of $\tilde X_2$ under any $p_i$. Let $Div_h(\tilde
X_1)\subseteq Div_g(\tilde X_1)$ denote the subgroup of divisors which
are trivial in $H^2(\tilde X_1,\Z)$.  Let
$R(\tilde X_1)$ be the product of the fields of  rational functions on the
connected components of $\tilde X_1$, and let $R(\tilde X_1)^*$ denote
the group of units. Define $R_g(\tilde
X_1)^*\subseteq R(\tilde X_1) ^*$ to
be the subgroup of functions whose divisor lies in $Div_g(\tilde
X_1)$. Let $\C(\tilde X_1)^*\subset R_g(\tilde X_1)^*$ denote the
subgroup of locally constant functions.
By an easy moving argument, we can see that the sequence 
$$0\to \C(\tilde X_1)^*\to R_g(\tilde X_1) ^*\to Div_g(\tilde X_1)\to
Pic^0(\tilde X_1)\to 0$$
is exact.

We now assume that $X$ is a surface. Then either using  remark~\ref{rmk:1}
or proposition~\ref{prop:gnpp}, we can see that $\tilde X_2$ can be
chosen to be zero dimensional. We do so. Let $\partial$ denote
the multplicative simplicial coboundary. 
Then the quotient  
$$R_g(\tilde X_1)^*/\partial^{-1}\partial\C(\tilde X_1)^*\cong \C(\tilde X_2)^*/\partial\C(\tilde X_1)^*$$
is  a finite dimensional multiplicative torus.
Following Carlson \cite{carlson2}, we define the group
$$P(\tilde X_\dt) = Div_h(\tilde X_1)/\partial^{-1}\partial\C(\tilde X_1)^*$$
which is an extension of $Pic^0(\tilde X_1)$ by the torus
$\C(\tilde X_2)^*/\partial\C(\tilde X_1)^*$. We remark that if we
allow $\tilde X_2$ to have positive dimensional components,  then
$P(\tilde X_1)$  is the wrong object to work with as it could be
infinite dimensional (\cite{carlson2} is not very explicit about this issue).
Let $D$ be a divisor class on  $\tilde X_0$ giving a homologically
Cartier cycle on $X$. Then  $\partial D\in Div_h(\tilde X_1)$ by definition.
So we get an induced map $\partial: Pic^0(\tilde X_0)\to P(\tilde X_\dt)$.

\begin{prop}\label{prop:carlson}
  With the above assumption that $\dim \tilde X_2=0$, we have  an isomorphism
  \begin{equation}
    \label{eq:C1}
    Ext_{MHS}^1(\Z(-1), W_1H^2(X,\Z))\cong P(\tilde X_\dt)/\partial
Pic^0(\tilde X_0)
  \end{equation}
Let $\alpha$ denote a homologically Cartier element in $\tilde H^2(X)$, and
let $D$ be a divisor  on $\tilde X$ representing it (which exists by
the Lefschetz $(1,1)$ theorem).
Then  $\varepsilon(\alpha)=0$ 
if and only if the image of $\partial D$ under the isomorphism
\eqref{eq:C1} vanishes.

\end{prop}

\begin{rmk}
This statement is sufficient for our purposes, although presumably $\varepsilon(\alpha) =\pm\im \partial D$. 
\end{rmk}

\begin{proof}
By a theorem of Deligne \cite[\S 10]{deligneH}, the category of polarizable
mixed Hodge structures of type  $\{(0,0), (1,0), (0,1), (1,1)\}$ is
equivalent to the category of $1$-motives. Let $H\subseteq H^2(X)$ be
the maximal submixed Hodge of this type.  A theorem of
Carlson  \cite[thm A]{carlson2} says that $H$ 
  corresponds to the $1$-motive 
$$\partial: NS(\tilde X_0)\to P(\tilde X_\dt)/\partial
Pic^0(\tilde X_0) $$
Under this identification, the mixed Hodge structure $E\subset H$
given by the extension class $\varepsilon(\alpha)$:
$$
\xymatrix{
 0\ar[r] & W_1H^2(X,\Z)\ar[d]^{=}\ar[r] & E\ar[r]\ar[d] & \Z(-1)\ar[r]\ar[d]^{[D]} & 0 \\ 
 0\ar[r] & W_1H^2(X,\Z)\ar[r] & H\ar[r] & \tilde H^2(X,\Z) & 
}
$$
corresponds to the sub motive
$$\Z(-1)\stackrel{\partial D}{\to} P(\tilde X_\dt)/\partial
Pic^0(\tilde X_0) $$
The proposition is an immediate consequence.
\end{proof}

The next example is a variation on one due to Totaro \cite{totaro}.

\begin{ex}\label{ex:2}
  Let $X$ be a normal surface constructed as in example \ref{ex:1}, but
  with $C$ a nodal cubic.  We build a simplicial scheme
$$ *\stackrel{\to}{\rightrightarrows} \tilde X_1=\PP^1\coprod
p\rightrightarrows \tilde X_0=\tilde X\coprod p\to X$$
where the maps are built from inclusions, projections and the
normalization of $\tilde C$. In this case,  $\varepsilon_1(D)=0$
because $J(\tilde X_1)=0$, but the class of $D$ in
  $P(\tilde X_\dt)=\C^*$ is non torsion, so that $\varepsilon(D)\not=0$.
\end{ex}

The final example, which is a variation on one due to Barbieri Viale
and Srinivas \cite{bs}, gives an example of a homologically Cartier
cycle which is not Cartier.

\begin{ex}\label{ex:3}
  Let $X$ be a normal surface constructed as in example \ref{ex:1}, but
  with $C$ a cuspidal cubic.  We proceed as above, but now $ P(\tilde X_\dt)=0$, so $\varepsilon(D)=0$. But $D$ has nontrivial class in
  $Pic^0(C)=\C$, so it cannot be Cartier. Note that in this case, Hodge theory
  is too coarse to detect the Picard group.
\end{ex}
\section{Two false starts}

This section contains two initial attempts by the author to answer the main question about where Hodge cycles come from.
Although neither gives the correct answer, we have included this
material  because  we feel that it is  nevertheless instructive. 

\subsection{Fulton's original Chow ring}
Fulton \cite{fulton1} defined a Chow ring for projective variety as
the limit
$$CH_F^*(X)= \varinjlim CH^*(Y)$$
where $X\to Y$ varies over all maps to smooth projective varieties. There
is an isomorphism $K^0(X)_\Q\cong CH^*(X)_\Q$, where
$K^0(X)$ is the Grothendieck group of vector bundles. 
Then there is cycle map $CH^*_F(X)_\Q \to H^{2*}(X)$, which can be
identified with Chern character. Thus the  image  of this map lies
in the space of Cartier cycles and therefore Hodge cycles. However, as
we have seen in example \ref{ex:3}, Hodge cycles need not be Cartier.
Many other examples can be found in \cite{ack}.  Therefore the cycle
map on $CH^*_F(X)_\Q $ is not surjective in general.

\subsection{The operational Chow ring}\label{sect:opchow}
 Let $CH_{OP}^*(X)$ denote the
operational Chow ring of Fulton-Macpherson \cite[chap 17]{fulton}. 
An element 
 $\alpha\in CH_{OP}^p(X)$ is a collection of operators
 $``f^*\alpha\cap": CH_*(X')\to CH_{*-p}(X')$ varying over morphisms  $f:X'\to X$.
 These are required to commute with  pushforwards,  flat pullbacks, 
 and Gysin maps  in the sense of \cite[def 17.1]{fulton}. 
This is an associative graded
ring which is contravariant, and has cap products. Furthermore, when
$X$ is smooth and $n$ dimensional, the operational Chow ring is isomorphic to the usual
Chow ring $\bigoplus_i CH_{n-i}(X)$ with the intersection product.

\begin{thm}[Kimura {\cite[thm 2.3]{kimura}}]
Let $X$ be projective variety, and let $\pi:\tilde X\to X$ be a
 resolution of singularities. Then
$$0\to CH_{OP}^p(X)\stackrel{\pi^*}{\longrightarrow}  CH^p(\tilde X)
\stackrel{p_1^*-p_2^*}{\longrightarrow}  CH_{OP}^p(\tilde X\times_X \tilde
X)$$
is exact, where $p_i:\tilde X\times_X \tilde X\to \tilde X$ denote the projections.
\end{thm}

\begin{cor}\label{cor:kim1}
 Suppose that $\tilde X_1\to \tilde X\times_X \tilde X$ is a
 resolution and $\tilde X_1'$ is as in remark~\ref{rmk:1}. Then we have exact sequences 
$$0\to CH_{OP}^p(X)\to  CH^p(\tilde X)
\to CH^p(\tilde X_1)$$
and
$$0\to CH_{OP}^p(X)\to  CH^p(\tilde X)
\to CH^p(\tilde X_1')$$
\end{cor}

\begin{cor}
 There is a natural ring homomorphism $CH_{OP}^*(X)_\Q\to C^*(X)$, which coincides with the usual cycle map, when $X$ is smooth.
\end{cor}

\begin{proof}
This follows from the diagram
 $$
\xymatrix{
 0\ar[r] & CH_{OP}^p(X)_\Q\ar[r]\ar@{-->}[d] & CH^p(\tilde X)_\Q\ar[r]\ar[d] & CH^p(\tilde X_1)_\Q\ar[d] \\ 
 0\ar[r] &  C^{p}(X)\ar[r] & C^{p}(\tilde X)\ar[r] & C^{p}(\tilde X_1)
}
$$
\end{proof}

We can see from the previous results that $\varepsilon_1(\alpha)=0$,
when $\alpha$ comes from $CH_{OP}^*(X)$. However, $\varepsilon(\alpha)$
need not be zero. To see this, we can use  the class $\alpha=D$ of
example \ref{ex:2}. Applying corollary
\ref{cor:kim1} with $\tilde X_1= \tilde X\coprod \PP^1\times \PP^1$
(see remark \ref{rmk:1}),
shows that $D$ lies in the image of $CH_{OP}^1(X)$.

The problem with the operational Chow group is that it is too
permissive. We need to restrict the classes so as to kill the
higher obstructions.
If $\alpha\in CH^p_{OP}(X)$, then in the above notation,
it corresponds to a cycle $\alpha_0$ on $\tilde X_0$ such that
the difference of the pullbacks $\partial\alpha_0=0$ in $CH^p(\tilde
X_1)$. This means that $\partial\alpha_0$ is the boundary of a higher
cycle $\alpha_1$ in the sense of Bloch (recalled below). If we insist
that $\alpha_1$ can be chosen so that $\partial \alpha_1$ is a
boundary in the Bloch complex of $\tilde X_2$, we get an additional constraint on $\alpha$
which is sufficient to prove $\varepsilon_2(\alpha)=0$. Continuing in
this way eliminates all the obstructions.  The precise statement is theorem \ref{thm:obstrvan}.
But first we need to recall basic facts about motivic cohomology.

\section{Motivic cohomology}\label{section:mot}

We start by recalling Bloch's complex  \cite{bloch}. 
 Let 
 $$\Delta^m = \Spec \C[x_0,\ldots, x_m]/(\sum x_i-1)\cong \A^m$$
be the algebraic geometer's simplex. Setting some the variables to $0$ defines the faces, which
can be labelled in the usual way.
Given a variety $Y$, let $Z_s^p(Y, -n)$ denote the space of codimension $p$ cycles on $Y\times \Delta^n$
meeting the faces properly. The coboundary 
\begin{equation}
  \label{eq:delta}
\delta \alpha = \sum (-1)^i \alpha\cap (i\text{th face})  
\end{equation}
turns this into a complex, which is Bloch's complex. The homology of this gives the higher Chow groups
$$ CH^p(Y, n) = H^{-n}Z^p(Y, \dt)$$
When $n=0$, this coincides with the usual Chow group. 
We also recall the cubical versions of this, referring to Levine \cite[\S 4]{levine} for details. 
Let $\square =
(\PP^1-\{1\})^n$ with coordinates $z_i$. Setting $z_i=0 $ or $\infty$
give the faces $\iota_{i,0},\iota_{i,1}:\square^{n-1}\to \square^n$.
Given a smooth projective variety $Y$, let  $Z_c^p(Y, -n)$ be the
 the quotient of the space of codimension $p$ algebraic cycles on $X\times \square^n$ meeting
  intersections of faces properly by the subspace of degenerate cycles.
This becomes a complex with differential
$$\delta = \sum (-1)^{i+j} \iota_{i,j}^*:Z_c^p(Y,-n)\to  Z_c^p(Y,-n+1)$$
This complex is quasi isomorphic to $Z^p_s(X, \dt)$.
When working rationally with  $Z_c^p(Y, -n)\otimes \Q$, we may also use the subcomplex of alternating cycles.
This eliminates the need to divide by degenerate cycles.

Given a finite collection of subvarieties $W_i\subset Y$,
one can form a subcomplex 
$$Z_s(Y,\dt)_{\{W_i\}}\subseteq Z_s(Y,\dt),$$
of cycles meeting the $W_i$ properly (and likewise for the cubic complexes).
Following  Hanamura, we  call these distinguished subcomplexes. The following
hold.
\begin{itemize}
\item The above inclusions are all quasiisomorphisms.
\item The intersection of two distinguished subcomplexes is
  distinguished. 
\item Given a morphism $f:Z\to Y$ between smooth
varieties, and a distinguished subcomplex $Z_s^p(Z,\dt)'$, there is
a distinguished subcomplex  $Z_s^p(Y,\dt)'$ such that pull back of
cycles gives a map
$f^*Z_s^p(Y,\dt)'\to Z_s^p(Z, \dt)'$ of complexes.
\end{itemize}
These properties ensure that $CH^p(-, n)$ is a covariant functor on 
the category of smooth varieties.

An alternative approach to  the higher Chow groups is to identify them with the 
cohomology of a complex of sheaves following
 Friedlander, Suslin and Voevodsky  \cite{fv, sv}.
Given a scheme $X$, we  recall  two relatively new
Grothendieck topologies. The first is the Nisnevich topology where the
covers are \'etale covers $U_i\to X$ such that for every possibly
nonclosed $x\in X$,
there is a $u$ in some $U_i$ lying over it with the same  residue
field $k(u)=k(x)$. For the cdh topology we also allow covers of the form
$$\tilde X \coprod Z\to X$$
where these form  a blow up square \eqref{eq:rs} but we allow $\tilde X$ to be singular.

Given schemes $U, V$, let $z_{qf}(V)(U)$ be the group of
 correspondences which are quasifinite over $U$; more precisely, it is the abelian group generated by irreducible subvarieties
 $V\times U$ which are quasifinite over $U$.  The group
 $z_{qf}(V)(-)$ is contravariant under pull back of correspondences, so it
determines a presheaf on  $X_{cdh}$.
 For each integer $n\ge 0$,  define the complex of presheaves $\Z_{X}(n)=\Z(n)$ on $X_{cdh}$ which assigns to a cdh
 open $U$,
 $$   \ldots\to \underbrace{z_{qf}(\A^n\times\Delta^1)(U)}_{\text{ deg $2n-1$}}\to
 \underbrace{z_{qf}( \A^n)(U)}_{\text{ deg $2n$}}$$
 with  coboundary $\delta$ as in \eqref{eq:delta}. This is similar to Bloch's
 complex, and in fact we have inclusions
 \begin{equation}
   \label{eq:SFtoB}
   \Z(n)(X)\subset Z_s^n(X\times \A^n,\dt)[-2n]
 \end{equation}
\cite[lemma 19.4]{mvw}; moreover, the image lies in any
distinguished subcomplex. Let $\Z(n)_{cdh}$, respectively
$\Z(n)_{zar}$, denote the sheafification of $\Z(n)$ in the cdh, respectively
Zariski, topologies.

Motivic cohomology is defined as
$$H_M^i(X,\Z(j)) := H^i(X_{cdh}, \Z_X(j)_{cdh})$$
and
$$H_M^i(X,\Q(j)) := H_M^i(X, \Z(j))\otimes \Q$$
 We make a few comments about this definition.

\begin{enumerate}

\item Since $\Z(j)$ is not bounded below, some care needs to be taken
  in defining hypercohomology. Given a complex of sheaves $S^\dt$ on a
  site with an element  $U$, we take
$$H^i(U, S^\dt) = H^i(\Gamma(U,\mathcal{I}^\dt))$$
where $\mathcal{I}^\dt$ is a $K$-injective resolution of $S^\dt$
in the sense of \cite{ajs, spaltenstein}. When $S^\dt$ is bounded
below, this coincides with the usual definition using injective resolutions.

\item The complex $\Z(j)$, which is denoted by $\Z^{SF}(j)$ in \cite{mvw}, is more convenient for our purposes
than the definition in lecture 3 \cite{mvw}. The two complexes are
quasi-isomorphic \cite[thm 16.7]{mvw}.

\item  The definition of  motivic cohomology as above, using the  cdh
  topology is taken from  \cite[ def 4.3, def 9.2] {fv}. When $X$ is
  smooth, it is possible and more convenient to work in the Zariski
  topology in the sense that
$$H_M^i(X,\Z(j))\cong  H^i(X_{zar}, \Z_X(j)_{zar})$$
cf \cite[ thm 5.5]{fv}. In the smooth case, motivic cohomology can be identifed
with higher Chow groups after reindexing, cf \cite{mvw} or theorem \ref{thm:motivic}.

\item There are products
$$H^i_M(X,\Z(j))\otimes H^{i'}_M(X,\Z(j'))\to H^{i+i'}_M(X, \Z(j+j'))$$
which agree with the natural products on higher Chow groups when $X$ is smooth,
cf  \cite[p 24]{mvw}, \cite{weibel}.

\end{enumerate}

Since we use the cdh topology, we get the following Mayer-Vietoris
sequence.
\begin{prop}\label{prop:cdh}
  Given the blow up square \eqref{eq:rs}, 
$$\ldots \to H_M^i(X,\Z(j)) \to H_M^i(\tilde X,\Z(j))\oplus
H_M^i(Z,\Z(j))\to H^i_M(E,\Z(j))\to\ldots $$
is exact.
\end{prop}

\begin{proof}
  This follows from \cite[ prop 4.3.3]{sv}.
\end{proof}

\section{Motivic cohomology via simplicial resolutions}

The definition of motivic cohomology given in the previous section is not terribly convenient for our purposes. 
Instead we will use the approach due to Hanamura \cite{hanamura}, using simplicial resolutions. 

To begin with, we need the existence of finite resolutions.

\begin{prop}[{\cite[chap 1, thm 2.6] {gnpp}}]\label{prop:gnpp}
  Given an $n$ dimensional quasiprojective variety
$X$, we can choose a smooth (semi-) simplicial scheme  with a
projective augementation $\tilde
X_\dt\to X$ satisfying cohomological descent, such that
 $\dim \tilde X_i\le n-i$ and in particular $\tilde X_i=\emptyset$ for $i>n$. 
\end{prop}

It will be useful to recall the basic idea of the construction of 
$\tilde X_\dt$, since it gives slightly more information than what is
stated above.
We will refer to any simplicial scheme constructed by
  this method as a {\em GNPP resolution} of $X$.

\begin{proof}  We
use  the  simplicial rather than cubical viewpoint of the original source.
As a first step, choose
a resolution of singularities $\pi:\tilde X\to X$ and a proper closed set $Z\subset X$ such
that $\pi$ is an isomorphism over $X-Z$.  Consider the diagram
\begin{equation}
  \label{eq:rs}
\xymatrix{
 E=f^{-1}Z\ar[r]\ar[d] & \tilde X\ar[d] \\ 
 Z\ar[r] & X
}
\end{equation}
which we refer to as a {\em blow up square}.
If $Z$ and $E$ are
both nonsingular, then we simply take  $\tilde X_0 = \tilde X\coprod
Z$, $\tilde X_1 = E$ and $\tilde X_2,\ldots=\emptyset$. This has an  obvious augmentation to $X$. In general, it is used as the foundation
for a more elaborate simplicial object  constructed inductively by
gluing appropriate GNPP resolutions of $Z$ and $E$.
So by construction, a GNPP resolution of $X$ can be decomposed 
as a disjoint union
$$\ldots \tilde X_1= \tilde E_0\coprod \tilde Z_1 \rightrightarrows \tilde X_0= \tilde X\coprod \tilde
Z_0 \to X$$
such that $\tilde Z_\dt$ (with solid arrows on the bottom of \eqref{eq:gnpp}) is a GNPP resolution of $Z$ and $\tilde E_\dt$ 
(with solid arrows on the top) is a GNPP resolution of $E$.
\begin{equation}
  \label{eq:gnpp}
\xymatrix{ 
\ldots & \tilde E_1\ar[r]\ar@<1ex>[r]\ar@{-->}[rd] & \tilde E_0\ar@{-->}[r]\ar@{-->}[rd] & \tilde X\ar@{-->}[rd] \\
\ldots & \tilde Z_2\ar[r]\ar@<1ex>[r]\ar@<-1ex>[r] & \tilde Z_1\ar[r]\ar@<1ex>[r] & \tilde Z_0\ar@{-->}[r] &X
}
\end{equation}
Furthermore, the rightmost parallelogram should map
to \eqref{eq:rs}. 
\end{proof}

  Given $X$, choose a  GNPP resolution $\tilde X_\dt$ as above. Next,  choose distinguished sub complexes $Z_s^p(\tilde X_\dt, \dt)'$
stable under any composition of  face maps. We let $\cX$ denote both sets of choices.
Then we can form a double complex
$(Z_s^p(\tilde X_\dt, \dt)', \delta, \pm \partial)$ and define
$$CH_H^p(X, n; \cX)= H^{-n}(Tot(Z_s^p(\tilde X_\dt, \dt)')$$
Hanamura \cite{hanamura} proves that, after tensoring with $\Q$, this  is independent of the choice of 
$\cX$. The next theorem will give a different proof  of this fact, which works integrally.

\begin{thm}\label{thm:motivic}
For any quasiprojective variety, there is an isomorphism
  $$H_M^m(X,\Z(n))\cong CH_H^n(X,2n-m; \cX)$$
  This is natural in the sense that given $Y\to X$ and a morphism of GNPP resolutions fitting into a commutative
  diagram
  $$
  \xymatrix{ \tilde Y_\dt\ar[r]\ar[d] & \tilde Y_\dt\ar[d]\\
   Y\ar[r] & X
   }
  $$
  there exists a commutative diagram
 $$\xymatrix{H_M^m(X,\Z(n))\ar[d]\ar^{\sim}[r] &CH_H^n(X,2n-m; \cX)\ar[d]\\
  H_M^m(Y,\Z(n))\ar^{\sim}[r] & CH_H^n(Y,2n-m; \mathcal{Y})
  }
  $$
  for some appropriate $\mathcal{Y}$.
\end{thm}

The proof of the theorem will be broken into a series of lemmas.  
We start with  the following.

\begin{lemma}\label{lemma:homotop}
  For any $i$,  there is a natural choice $\cX'$ for which we have a canonical 
  isomorphism $CH_H^n(X, 2n-m;\cX) \cong CH_H^n(X\times
  \A^i,2n-m;\cX')$.
  \end{lemma}

\begin{proof}
The simplicial scheme $\tilde X_\dt\times \A^i\to X\times \A^i$ is a
GNPP resolution mapping to $\tilde X_\dt$. We pullback the distinguished complexes to $\tilde X_\dt\times \A^i$.
These choices will be denoted by $\cX'$.
We have  a map of double complexes, and hence  a morphism of spectral sequences
$$
\begin{array}{ccc}
 CH^p(\tilde X_a, -b) & \Rightarrow & CH_H^p(X, a-b; \cX) \\ 
 \downarrow &  & \downarrow \\ 
 CH^p(\tilde X_a\times \A^i, -b) & \Rightarrow & CH_H^p(X\times \A^i, a-b; \cX')
\end{array}
$$
By \cite[thm 2.1]{bloch}, the vertical maps on the left are isomorphisms,
therefore the vertical map on the right is also an isomorphism.
\end{proof}

 We can form the category, in fact topos,  $Sh(\tilde X_{cdh,\dt})$  where an object consists of a collection of cdh sheaves $\F_i$ on $\tilde
  X_i$ and face maps $p_j^*\F_i\to \F_{i+1}$. Morphisms $\F_\dt\to
  \F_\dt'$  are
  collections of morphisms of sheaves compatible with the face maps.
The complex $\Z_{\tilde X_\dt}(n)_{cdh}$ 
can be regarded as a complex in this category. Choose a $K$-injective resolution $\mathcal{I}^{\dt\dt}$ of 
$\Z_{X_\dt}(n)_{cdh}$,  where the second index  on  $\mathcal{I}^{\dt\dt}$ is the simplicial
  index. Let $I^{ab} = H^0(\tilde X_b,\mathcal{I}^{ab})$. We have
  quasiisomorphisms  of complexes
$$Z_s^n(\tilde X_j\times \A^n,\dt)'[-2n]\leftarrow \Z(n)(\tilde X_j)\to H^0(\tilde X_j,
\Z_{\tilde X_j}(n)_{cdh})\subset I^{\dt j}$$
by \cite[thm 19.1]{mvw} (and its proof). These quasiisomorphisms are
compatible with the coboundary operator $\delta$. Thus we have a
quasiisomorphism of the total complexes. Consequently, we have
$$H^m(Tot( I^{\dt\dt}))\cong H^{m-2n}(Z_s^n(\tilde X_\dt\times \A^n,\dt)')\cong CH_H^n(X,2n-m;\cX)$$
So to finish the proof of theorem \ref{thm:motivic}, we need.

\begin{lemma}
  $H^m(Tot( I^{\dt\dt}))\cong H^m(X_{cdh},\Z(n)_{cdh})$.
\end{lemma}

\begin{proof}
The first group  $H^m(Tot( I^{\dt\dt}))$ is nothing but the cohomology
of $\Z_{\tilde X_\dt}(n)_{cdh}$ in the topos $Sh(\tilde X_{cdh,\dt})$. We
have a morphism of topoi $Sh(\tilde X_{cdh,\dt})\to Sh(X_{cdh})$,
induced by $\pi_\dt$, which  induces a morphism of groups
$$\gamma^m:H^m(X_{cdh},\Z(n)_{cdh})\to H^m(\tilde X_{cdh, \dt}, \Z(n)_{cdh})$$
(More explicitly, choose a $K$-injective resolution $\mathcal{J}^\dt$
of $\Z_X(n)$, then
we can easily construct a map of complexes $H^0(X,\mathcal{J}^\dt)\to Tot(I^{\dt\dt})$
inducing the above map.) Will prove that this map is an isomorphism
by induction on the length of $\tilde X_\dt$. By
  length, we mean the smallest integer $d$ such that $\tilde
  X_{d+i}=\emptyset$ for all $i>0$. As in the proof of proposition
  \ref{prop:gnpp}, we can assume that $\tilde X_\dt$ has the structure
  given in \eqref{eq:gnpp}. Using this we get a commutative diagram
$$
\xymatrix{
 \ldots & H^{m-1}(E)\ar[r]\ar[d]^{\alpha^{m-1}} & H^m(X)\ar[r]\ar[d]^{\gamma^m} & H^m(\tilde X)\oplus H^m(Z)\ar[d]^{\beta^m} & \ldots \\ 
  & H^{m-1}(\tilde E_\dt)\ar[r] & H^m(\tilde X_\dt)\ar[r] & H^m(\tilde X)\oplus H^m(\tilde Z_\dt) & 
}
$$
where the spaces $E$ etc. are the same as in the proof of  proposition
  \ref{prop:gnpp}, and the coefficients are $\Z(n)_{cdh}$. By
  induction, the arrows $\alpha^\dt,\beta^\dt$ are
  isomorphisms. Therefore $\gamma^\dt$ is an isomorphism by the $5$-lemma.

\end{proof}

\begin{rmk}\label{rmk:ncd}
  When $X= X_1\cup X_2$ is a union of smooth varieties meeting
  transversally, the GNPP algorithm produces
$$ X_1\cap X_2\coprod X_1\cap X_2\rightrightarrows X_1\coprod X_2\coprod X_1\cap X_2\to X$$
In this case, it is more efficient to cancel one of the $X_1\cap X_2$
factors and use
$$ X_1\cap X_2\rightrightarrows X_1\coprod X_2\to X$$
It is not difficult to see that the resulting double complexes are quasiisomorphic.
More generally if $X=\cup X_i$ has global normal crossings, i.e. when the
  components and their intersections are smooth of expected dimensions, rather than using a
  GNPP resolution, we can  substitute the simpler
  simplicial resolution
$$\ldots \coprod X_i\cap X_j\rightrightarrows \coprod X_i\to X$$
in theorem~\ref{thm:motivic}. 
\end{rmk}

\section{Motivic classes are unobstructed }

Let $\pi:\tilde X\to X$ be a desingularization. Then we have an induced map
$$\pi^*: H_M^{2p}(X,\Z(p))\to H_M^{2p}(\tilde X, \Z(p))\cong CH^p(\tilde X)$$

\begin{prop}
The image  $\pi^*(H^{2p}_M(X,\Z(p)))\subseteq CH^*(\tilde X)$
depends only $X$ and  lies in $CH^p_{OP}(X)$.
\end{prop}

\begin{proof}
The proof that the image is well defined is similar to the proof of lemma \ref{lemma:W}. So we focus on
the last part. We extend $\tilde X\to X$ to a GNPP resolution $\tilde
X_\dt$ as in the proof of proposition~\ref{prop:gnpp}, using the same
notation as in that proof. So $\tilde X_\dt$ takes the form
$$\ldots \tilde X_1= \tilde E_0\coprod \tilde Z_1 \rightrightarrows \tilde X_0= \tilde X\coprod \tilde
Z_0 \to X$$
Let $Y$ be a desingularization of $\tilde E_0\times_{\tilde Z_0} \tilde E_0$.
We have a commutative diagram
$$
\xymatrix{
 Y\ar@/^/[r]^{p_1}\ar@/_/[r]^{p_2}\ar[rd]^{g} & \tilde E_0\ar[r]^{i}\ar[d]^{f} & \tilde X \\ 
  & \tilde Z_0 & 
}
$$
where the arrows are the obvious projections.
 
From the double complex $Z_s^p(\tilde X_\dt, \dt)'$, we get a fourth quadrant spectral sequence
\begin{equation}\label{eq:motss}
 E_1^{ab} = CH^p(\tilde X_a, -b)\cong H^{2p+a}_M(X_a,\Z(p)) \Rightarrow CH_H^p(X, a-b) \cong H^{2p-a+b}_M(X,\Z(p))
\end{equation}
This certainly does depend on the
choice of $\tilde X_\dt$. However, the image of edge map
$$H_M^{2p}(X,\Z(0))\to E_\infty^{00}\subseteq\ldots E_2^{00}\subseteq E_1^{00}= CH^*(X)$$
does not, because it is just $\im \pi^*$.
To complete the proof, we will show that
$$E_2^{00}=
\ker[CH^p(\tilde X_0)\to CH^p(\tilde X_1)]$$
lies in  $CH^p_{OP}(X)$. If $\alpha\in E_2^{00}$, then $i^*\alpha=0$ in
$CH^*(\tilde E_0)/\im CH^*(\tilde Z_0)$ or equivalently  that $i^*\alpha=
f^*\beta$ for some $\beta$. Therefore $(i\circ p_1)^*\alpha -(i\circ
p_2)^*\alpha = g^*\beta-g^*\beta=0$. Now applying 
corollary~\ref{cor:kim1} implies that $\alpha\in CH_{OP}^*(X)$ ($Y$
plays the role of $\tilde X_1'$ in  said corollary).
\end{proof}

Thus we get a cycle map given by the  composition
\begin{equation}\label{eq:cyclemap}
 H^{2p}_M(X,\Q(p))\stackrel{\pi^*}{\to} CH^p_{OP}(X)_\Q\to  \tilde H^{2p}(X,\Q(p))
\end{equation}
We will prove that if $\alpha\in H^{2p}_M(X,\Q(p))$, the obstruction $\varepsilon(\alpha)$ of the
  image of this  class in $\tilde H^{2p}(X)$ vanishes. Therefore the image of  $\alpha$ lifts to a Hodge
cycle in $H^{2p}(X)$, which by lemma \ref{lemma:HCinj} is unique. In fact, the statement we prove  is a bit more precise.

\begin{thm}\label{thm:obstrvan}
There is a homomorphism
$$H^{2p}_M(X,\Q(p))\to H^{2p}(X,\Q(p))$$
such that the composite with the projection to $\tilde H^{2p}(X,\Q)$ coincides with \eqref{eq:cyclemap}.
Furthermore the image of this map lands in the space of weight $2p$ Hodge cycles. 
 \end{thm}

The proof will be given below after the necessary preparation.

 Recall that if  $Y$ is an $n$-dimensional complex manifold, the space of
degree $p$ currents $\D^p(U)$, over an open set $U\subseteq Y$, is the
topological dual of the space of compactly supported  
forms $\E_0^{2n-p}(U)$ cf \cite[chap 3 \S1]{gh}. These form a complex
of fine sheaves.  We denote the differential by $d$.
This admits a
bigrading into $(p,q)$ type, and therefore a Hodge filtration. 
Any $(p,q)$ differential form $\alpha$ with locally $L^1$-coeffients
defines an element $\cC(\alpha)\in \D^{p,q}(Y)$ given by $\phi\mapsto
\int_Y\alpha\wedge \phi$. When $\alpha$ is $C^\infty$, we  will usually just conflate $\alpha$ and $\cC(\alpha)$. 
However, it is a good idea to  maintain a distinction  when $\alpha$
is singular because  certain operations such as $d$ do not commute
with $\cC$ in general.
Any piecewise smooth oriented chain $(2n-p)$-chain
$\Gamma$ in $Y$, defines an element  $T_\Gamma\in \D^p(Y)$ given
by $T_\Gamma(\phi)=\int_\Gamma\phi$. Let $\D^p_\Q(Y)$
denote the span of such chains with rational coefficients.
We give $\D^p(Y)$ the weak topology: $\eta_i\to
\eta$ if $\eta_i(\phi)\to\eta(\phi)$ for every $\phi\in
\E_0^{2n-p}(U)$.  Smooth forms are dense. Pushforwards of currents
under proper $C^\infty$-maps are defined as adjoint to pullbacks of
compactly supported forms. Pull backs are more delicate.
If $f:Y'\to Y$ is a $C^\infty$ map, and
$\eta\in \D^p(Y)$, we say that the pullback $f^*\eta$ {\em exists  and is
equal to a current} $\xi\in \D^p(Y')$ if there
is a sequence $\eta_i\in \E^p(Y)$ converging to $\eta$, such
that $f^*\eta_i\to \xi$. We note that $f^*\eta$ need not exist in general.
This definition is implicit in a theorem of H\"ormander \cite[thm 8.2.4]{hormander},
which gives a very general criterion for the existence of pullbacks of
distributions. For our purposes, the following criterion is
sufficient and  easy to check.

\begin{lemma}
  Suppose that $\alpha$ is a locally $L^1$ differential form on  a smooth connected quasiprojective variety $Y$
such that $\alpha$ is $C^\infty$ off of a proper real semialgebraic set $T\subset Y$. If $f:Y'\to Y$ is morphism
from another  smooth connected quasiprojective variety such that $f(Y')\nsubseteq T$. Then
$f^*\cC(\alpha)$ exists and is given by $\cC(f^*\alpha)$.
\end{lemma}

 Let us say that $\alpha$ has {\em mild singularities} (along $T$) if the
 assumptions of the last lemma apply. 

\begin{lemma}\label{lemma:pushpull}
 Suppose that we have a commutative diagram
$$
\xymatrix{
 Z'\ar[r]^{F}\ar[d]^{P} & Z\ar[d]^{p} \\ 
 Y'\ar[r]^{f} & Y
}
$$
of smooth projective varieties and a current $\omega$ on $Z$
satisfying the following conditions:
\begin{enumerate}
\item $Z'$ is birational to $Z\times_Y Y'$.
\item $\omega$ has mild singularities along $T\subset Z$.
\item $F(Z')\nsubseteq T$
\item There is a Zariski closed set $\Delta\subset Y$ such that
  $f(Y')\nsubseteq p(T)\cup \Delta$ and $p$ is smooth over $Y-\Delta$.
\end{enumerate}
Then the currents $f^*p_*\omega$  and $ P_*F^*\omega $ both exist, are equal, and
have mild singularities.
\end{lemma}

\begin{proof}

As a first case, suppose that $T=\emptyset$ (so that $\omega$ is $C^\infty$),
 $p$ is smooth, and the diagram is Cartesian. 
Then $p_*\omega$ and $P_*F^*\omega$ are gotten by integration along
the fibres. An easy calculation shows  that integration along fibres
commutes with pullback, therefore $f^*p_*\omega= P_*F^*\omega $.

In the general case, we can assume that the assumptions of the first case
hold for the diagram
$$
\xymatrix{
 Z'-P^{-1}(f^{-1}(p(T)\cup\Delta)\ar[r]\ar[d]^{P} & Z-T\cup p^{-1}\Delta\ar[d]^{p} \\ 
 Y'-f^{-1}(p(T)\cup\Delta)\ar[r]^{f} & Y- p(T)\cup\Delta
}
$$
after enlarging $\Delta$ if necessary. Thus we have equality of forms
$f^*p_*\omega= P_*F^*\omega $  on the complement on
$f^{-1}(p(T)\cup\Delta)$.
The only additional thing  to observeis that the forms are  $L^1$ on $Y$. We can see this by observing
that $\int_{Y'} |P_*F^*\omega| dvol\le Const.\int_{Z'} |F^*\omega| dvol$.

\end{proof}

 Given a set of maps $F=\{f_i:Y_i\to Y\}$, let
$\D_F^p(U)\subset \D^p(U)$ be the set currents for which the pullback
 along $f_i\in F$ exists. This is easily seen to give a subsheaf of
 $C_Y^\infty$-modules.  Moreover, $d(\D^p_F)\subseteq \D^{p+1}_F$. 
Let   $\D_{F;\Q}(U)= \D_{F}(U)\cap \D_\Q(U)$.

The currents of interest to us were constructed by  Kerr, Lewis and
M\"uller-Stach \cite{klm}. Given subvariety  $Z\subset Y\times
\square^n$, we can pull back the coordinates $z_i$ on $\square^n$ to  functions on
$Z$. We will say that $Z$ is {\em admissible} if 
 $Z$ meets  all intersections of divisiors $(z_i)$ and faces properly,
 and in particular that $z_i$ does not vanish on $Z$. 
 Choose a desingularization of a
compactification $\tilde Z$ of $Z$, and pull back $z_i$ to this space.
We define the currents and cycles on $\tilde Z$ by
\begin{equation*}
  \begin{split}
    A'(\tilde Z) &= \cC(\frac{dz_1}{z_1}\wedge\ldots \wedge\frac{dz_n}{z_n})\\
 \Gamma_i &= z_i^{-1}[-\infty, 0] \text{ (as a cycle oriented from
   $-\infty$ to $0$)}\\
 B'(\tilde Z) &= T_{\Gamma_1\cap \ldots \cap\Gamma_n}\\
 C'(\tilde Z) &=\cC(\log z_1\frac{dz_2}{z_2}\wedge\ldots \frac{dz_n}{z_n}\pm (2\pi
 i) \log z_2 \frac{dz_3}{z_3}\wedge\ldots
 \frac{dz_n}{z_n}T_{\Gamma_1}+\\
& \ldots\pm (2\pi i)^{n-1} T_{\Gamma_1\cap \ldots \Gamma_{n-1}})
  \end{split}
\end{equation*}
The logarithm above is the branch with imaginary part in $(-\pi,\pi)$ on $\C-[-\infty,0]$. 
Now define
$$A(Z) = \pi_*A'(\tilde Z),\> B(Z) = \pi_* B'(\tilde Z), \>C(Z) =\pi_*
C'(\tilde Z)$$
where $\pi:\tilde Z\to Y$ is the projection.

Suppose that $Z=\sum n_iZ_i\in Z_c^p(Y,n)$ is a
cycle all of whose components are admissible, then we can 
extend the above definitions by linearity to obtain currents
$$A(Z)\in F^p\D^{2p-n}(Y),\>  B(Z)\in \D_\Q^{2p-n}(Y),\> C(Z) \in \D^{2p-n-1}(Y)$$

\begin{prop}[{\cite[(5.5)]{klm}}]\label{prop:klm}
  The following relations hold
$$d A(Z) = (2\pi i) A(\delta Z)$$
$$dB(Z) = B(\delta Z)$$
$$ d C(Z) = A(Z) - (2\pi i)^n B(Z) - 2\pi iC(\delta Z)  $$
\end{prop}

Finally note that \cite[\S 5.4]{klm} shows that the subcomplex of
admissible cycles $Z^p(Y,\dt)_{ad}\subset Z^p(Y,\dt)$ is quasiisomorphic to
the full complex. A similar argument applies to distinguished complexes.

\begin{proof}[Proof of theorem \ref{thm:obstrvan}]
We start by proving the weaker statement immediately preceding the theorem that if $\alpha\in H_M^{2p}(X,\Q(p))$,
then $\varepsilon(\alpha)=0$.
In order to calculate $\varepsilon(\alpha)$, we use a modification of the set up
described in the paragraph before proposition \ref{prop:epsilon1}.
We replace  the complex of differential forms $(\E^\dt(\tilde X_a),d)$ with the complex
of currents $(\D^\dt_P(\tilde X_a),d)$, where $P$  is the set of face maps
$p_b:\tilde X_{a+1}\to \tilde X_a$. This forms a double complex, with the
second differential given by the simplicial coboundary $\partial =
\sum (-1)^b p_b^*$.
Let $D^\dt$  and $D_\Q^\dt$ be the total complexes 
$\D^\dt_P(\tilde X_\dt)$ and $\D^\dt_{P;\Q}(\tilde X_\dt)$. Then we will choose our representatives
$A = (A_0, A_1,\ldots)\in F^pD^{2p}$ and $B=(B_0, B_1,\ldots) \in
D_\Q^{2p}$, and $\varepsilon(\alpha)$ will be represented by the
difference $A-B$.

The element $\alpha\in H_M^{2p}(X,\Q(p))$ can be represented by a collection of cycles
$$\alpha_a\in Z^p(\tilde X_a, a)_{ad}'$$
such that
\begin{equation}
  \label{eq:alpha}
  \begin{split}
  \delta\alpha_0 &=0\\
\partial \alpha_0&=\delta \alpha_1\\
\ldots
\end{split}
\end{equation}

We can associate the currents
$$A_a=(2\pi i)^{-a}A(\alpha_a)\in F^p\D^{2p-a}(\tilde X_a)$$
$$B_a = B(\alpha_a)\in \D_\Q^{2p-a}(\tilde X_a)$$
$$C_a= (2\pi i)^{-a}C(\alpha_a)\in \D^{2p-a-1}(\tilde X_a)$$
as above. By lemma \ref{lemma:pushpull}, we can see that the pull backs of these
currents along the face maps $p_b:\tilde X_{a+1}\to \tilde X_a$ exist, and
$$p_b^*A_a = A(p_b^*\alpha)$$
etc. It follows from this, proposition \ref{prop:klm}, and \eqref{eq:alpha} that
$$ dA_a = (2\pi i)^{1-a} A(\delta \alpha_a) = (2\pi i)^{1-a}A(\partial \alpha_{a-1})
=\partial A_{a-1}$$
$$ dB_a = B(\delta\alpha_a) = \partial B_{a-1}$$
$$ dC_a = A_a-B_a - \partial C_{a-1}$$
This implies that $(A_\dt)\in F^pD^\dt$  and $(B_\dt)\in D_\Q^\dt$ are
cocycles, and that
$(A_\dt-B_\dt)$ is a coboundary. This proves that $\varepsilon(\alpha)=0$. Now to get the full
statement, note that $\alpha\mapsto (2\pi i)^p(B_0,B_1,\ldots)$ determines homomorphism from
$H_M^{2p}(X,\Q(p))$ to the space of weight $2p$ Hodge cycles in $H^{2p}(X,\Q(p))$.

\end{proof}

\section{Lefschetz $(1,1)$ theorem}

We start with some explicit descriptions of degree two motivic
cohomology. Throughout this section, $X$ is a projective variety over
$\C$, with a GNPP resolution $\tilde X_\dt\to X$. The face maps are
denoted by $p_i$. We use the notation
$Div_g(\tilde X_i), R_g(\tilde X_i)^*$ for divisors or functions in
general position with respect to face maps introduced in \S \ref{sect:HC}.

\begin{prop}\label{prop:H2M}
An element of $H^2_M(X,\Z(1))$ is represented by
 a pair $(D,f)$, where $D\in Div_g(\tilde X_0)$ and $f\in R_g(\tilde X_1)^*$ such that
 \begin{equation}
  \label{eq:cobound}
  \begin{split}
\partial D &= (f)\\
 \partial f &=1  \text{ ($\partial$ on $R_g^*$ is  multiplicative)}
\end{split}
\end{equation}
Two pairs $(D_i, f_i)$ represent the same class if there exists $g\in
R_g(\tilde X_0)^*$ such that
\begin{equation*}
   \begin{split}
D_1-D_2 &= (g)\\
  f_1/f_2 &= \partial g  
\end{split}
\end{equation*}
\end{prop}

\begin{proof}
 Consider the diagram
$$
\xymatrix{
  & Z^1(\tilde X_0,0)'\ar[rr]\ar[ld]^{\phi} &  & Z^1(\tilde X_1, 0)'\ar[rr]\ar[ld]^{\phi} &  & Z^1(\tilde X_2, 0)'\ar[ld]^{\phi} \\ 
 Div_g(\tilde X_0)\ar[rr] &  & Div_g(\tilde X_1)\ar[rr] &  & Div_g(\tilde X_2) &  \\ 
  & Z^1(\tilde X_0,1)'\ar[rr]\ar[uu]\ar[ld]^{\phi} &  & Z^1(\tilde X_1,1)'\ar[rr]\ar[uu]\ar[ld]^{\phi} &  & Z^1(\tilde X_2,1)'\ar[uu]\ar[ld]^{\phi} \\ 
 R_g(\tilde X_0)^*\ar[rr]\ar[uu] &  & R_g(\tilde X_1)^*\ar[rr]\ar[uu] &  & R_g(\tilde X_2)^*\ar[uu] & 
}
$$
The  columns $R_g(-)\to Div_g(-)$ of the front face  have just the  two terms, but the remaining
columns $Z^1(\tilde X_\dt, \dt)$ may be longer. The diagonal arrows $\phi$ are constructed by
Nart \cite{nart}; they give quasi-isomorphisms between columns. Thus we
may use the total complex of  the front face to compute
$CH_H^1(X,*)$. In particular, this yields the description of $H_M^2(X,\Z(1))=CH^1_H(X,0)$ stated 
above.
\end{proof}

We can use this description to construct certain elements of motivic
cohomology.
 If $E$ is Cartier divisor on $X$, $(\pi^* E, 1)\in
 H^2_M(X,\Z(1))$. Of course,
 in general, there are additional elements in $ H^2_M(X,\Z(1))$. 
Given a simplicial scheme such as $\tilde X_\dt$, we can apply  the
connected components functor $\pi_0$ to get a simplicial set called
the dual complex $\Sigma$. Composing this with the free abelian group
functor gives a simplicial abelian group whose cohomology we denote by
$H^*(\Sigma,\Z)$. This is the same thing as the singular cohomology
Its geometric realization $|\Sigma|$.

\begin{cor}
  We have an exact sequence
$$1\to H^1(\Sigma,\Z)\otimes \C^*\to H_M^2(X,\Z(1))\to Pic(\tilde X_0)$$
The image of the
last map is the set of   classes of divisors $D$ satisfying
\eqref{eq:cobound} for some $f$.
\end{cor}

\begin{proof}
  We have a map $H_M^2(X,\Z(1))\to Pic(\tilde X_0)$ which sends
  $(D,f)$ to the class of $D$. It  can be checked that
  $\{(0,f)\mid f\in \C(\tilde X_1)^*, \partial f=1\}$ maps onto the
  kernel, and that the kernel of this map is precisely $\{(0,\partial
  g)\mid g\in \C(\tilde X_0)^*\}$.
\end{proof}

The proposition also leads to an interpretation of $H^2_M(X,\Z(1))$ as
line bundles on $\tilde X_0$ with descent data.

\begin{cor}
  Elements of $H^2_M(X,\Z(1))$ is the group of isomorphism classes of pairs $(L,
  \sigma:p_0^*L\cong p_1^*L)$, where $L$ is a line bundle on $\tilde
  X_0$  and $(p_{0}^*\sigma) (p_1^*\sigma)^{-1}(p_2^*\sigma)=1$.
\end{cor}

\begin{proof}
  Send $(D,f)$ to $(\OO(D), \OO(p_0^*D)\stackrel{f}{\to} \OO(p_1^*D))$.
\end{proof}

Let $H^1(\tilde X_\dt,\OO_{\tilde X_\dt}^*)$ denote the cohomology of the Zariski simplicial sheaf
$\OO_{\tilde X_\dt}^*$. From the spectral sequence
$$E_1^{pq}= H^q(\tilde X_p, \OO_{\tilde X_p}^*)\Rightarrow H^{p+q}(\tilde X_\dt,
\OO_{\tilde X_\dt}^*)$$
we get an exact sequence
$$ 0\to H^1(\Sigma, \Z)\otimes \C^*\to H^1(\tilde X_\dt,\OO_{\tilde X_\dt}^*)\to
H^1(\tilde X_0,\OO_{\tilde X_0}^*)$$
The image of the last map is precisely the $E_3^{01}$.

\begin{prop}
 There is a natural  isomorphism 
$$\eta:H_M^2(X,\Z(1))\cong H^1(\tilde X_\dt,\OO_{\tilde X_\dt}^*)$$
\end{prop}

\begin{proof}
Let $R_{g,X}^*$ denote the sheaf $U\mapsto R_{g}^*(U)$.
We can see that $\OO^*_{\tilde X_i}\subset R_{g,\tilde X_i}$, so there are exact sequences
$$1\to \OO^*_{\tilde X_i}\to R_{g,\tilde X_i}^*\stackrel{div}{\to}  R_{g,\tilde
  X_i}^*/\OO^*_{\tilde X_i}\to 1 $$
Thus we have a quasi isomorphism
 $$\OO^*_{\tilde X_\dt}\sim_{qi} R_{g,\tilde X_\dt}^*\to  R_{g,\tilde X_\dt}^*/\OO^*_{\tilde X_\dt} $$
 of complexes of simplicial sheaves. This  induces an isomorphism 
  $$H^1(\tilde X_\dt,\OO^*_{\tilde X_\dt})\cong H^1(\tilde X_\dt, R_{g,\tilde X_\dt}^*\to  R_{g,\tilde X_\dt}^*/\OO^*_{\tilde X_\dt} )$$
  We can  compute the right side as the cohomology of the 
  total complex
  $$T^n = \bigoplus_{i+j+k=n} I^{ijk}$$
  where the terms on the right are global sections of injective resolutions fitting into a diagram
$$
\xymatrix{
 R_{g,X_i}^*\ar[r]\ar^{div}[d] & \I^{0,i,0}\ar^{div}[d]\ar^{d}[r]  &\I^{0,i,1}\ar^{div}[d]\ar^{d}[r] &\\
 R_{g,X_i}^*/O_{X_i}^*\ar[r] & \I^{1,i,0}\ar^{d}[r]  &\I^{0,i,1}\ar^{d}[r] &
}
$$
We need to be a bit careful about the sign. The differential of $T$ on $I^{abc}$ is 
$\kappa+ (-1)^a\partial + (-1)^{a+b}d$.
 One sees that if $(D,f)\in H_M^2(X,\Z(1))$ as in proposition \ref{prop:H2M}, then 
 $$(D,f,0)\in T^1=I^{100}\oplus I^{010}\oplus I^{001}$$
 is a cocycle, and hence it defines a class
 $\eta(D,f)\in H^1(X_\dt,\OO_{X_\dt}^*)$. We see that this fits into a
 commutative diagram
$$
\xymatrix{
 1\ar[r] & H^1(\Sigma,\Z)\otimes \C^*\ar[r]\ar[d]^{=} & H_M^2(X,\Z(1))\ar[r]\ar[d]^{\eta} & Pic(\tilde X_0)\ar^{=}[d] \\ 
 1\ar[r] & H^1(\Sigma, \Z)\otimes \C^*\ar[r] & H^1(\tilde X_\dt,\OO_{\tilde X_\dt}^*)\ar[r] & H^1(\tilde X_0,\OO_{\tilde X_0}^*)
}
$$
This implies that $\eta$ is injective. We have to show that it is surjective. Suppose that
$(\alpha,\beta, \gamma)\in T^1$ is a cocycle. This implies that
\begin{equation}
\label{eq:cobound1}
d\gamma=0
\end{equation}

\begin{equation}
\label{eq:cobound2}
-d\beta+\partial\gamma=0
\end{equation}

\begin{equation}
\label{eq:cobound3}
-d\alpha+ \kappa\gamma=0
\end{equation}

\begin{equation}
\label{eq:cobound4}
\partial\beta=0
\end{equation}

Since $R_g^*$ is flasque, equation \eqref{eq:cobound1} implies that $\gamma = d\xi$ for some $\xi\in I^{000}$
After adding the coboundary corresponding to $-\xi$ to $(\alpha,
\beta, \gamma)$, the remaining equations imply that
it lies in the image of $H_M^2(X,\Z(1))$.
  \end{proof}

  \begin{cor}\label{cor:HM}
    There is a natural  isomorphism 
$$\eta:H_M^2(X,\Z(1))\cong H^1(\tilde X_{an,\dt},\OO_{\tilde X_{an,\dt}}^*)$$
where $\OO_{\tilde X_{an,\dt}} $  is the simplicial sheaf of
holomorphic functions.
  \end{cor}

  \begin{proof}
By standard arguments, we get a  map of spectral sequences

$$
\begin{array}{ccc}
 E_1^{pq}= H^q(\tilde X_p, \OO_{\tilde X_p}^*) & \Rightarrow & H^{p+q}(\tilde X_\dt,
\OO_{\tilde X_\dt}^*) \\ 
\downarrow & &\downarrow\\
 E_1^{pq}= H^q(\tilde X_{an,p}, \OO_{\tilde X_{an,p}}^*) & \Rightarrow & H^{p+q}(\tilde X_{an,\dt},
\OO_{\tilde X_{an,\dt}}^*)
\end{array}
$$
By  GAGA, we have an isomorphism of $E_1$'s and therefore of
abutments. The corollary follows from this and the proposition.
  \end{proof}

By weight $2$ Hodge cycles in $H^2(X,\Z(1))$ we simply mean the preimage of $F^1$ under the map
$H^2(X,\Z(1))\to H^2(X,\C)$. This includes all the torsion cycles.

\begin{prop}\label{prop:cH2M}
 There exists a natural homomorphism $c$ which fits into a commutative diagram
$$
\xymatrix{
 H_M^2(X,\Z(1))\ar[r]^{c}\ar[d] & H^2(X,\Z(1))\ar[d] \\ 
 Pic(\tilde X_0)\ar[r]^{c_1} & H^2(\tilde X_0,\Z(1))
}
$$
The image of $c$  lands in the space of weight 2 Hodge cycles.
\end{prop}

\begin{proof}
By corollary \ref{cor:HM}, we can and will identify motivic cohomology with the first cohomology
of $\OO_{\tilde X_{an, \dt}}^*$.
We have the exponential sequence
$$0\to \Z(1)\to \OO_{\tilde X_{an, \dt}}\stackrel{\exp}{\to}\OO_{\tilde X_{an, \dt}}^*\to 1$$
which yields an exact sequence
$$H^1(\tilde X_{an,\dt},\OO_{\tilde X_{an,\dt}}^*)\stackrel{c}{\to} H^2(X,\Z(1))\to
H^2(\tilde X_{an,\dt}, \OO_{\tilde X_{an,\dt}})$$
This implies that
$$\im c\subseteq H^2(X,\Z(1))\cap \ker [H^2(X,\C)\to (\tilde X_{an,\dt}, \OO_{\tilde X_{an,\dt}})]$$
By \cite[thm 4.5]{dubois}, this can be identified with $ H^2(X,\Z(1))\cap F^1$, where the intersection
is understood as the preimage. 
\end{proof}

\begin{cor}
  The map $c\otimes \Q$ coincides with the map constructed in theorem \ref{thm:obstrvan}.
\end{cor}

\begin{proof}
  The two maps induce the same map to $\tilde H^2(X,\Q(1))$, so they must
  coincide by lemma ~\ref{lemma:HCinj}.
\end{proof}

We now come to the main result of this section.

\begin{thm}\label{thm:Lef}
  Given a projective variety $X$,  the image of the  map
  $$c:H_M^2(X, \Z(1))\to H^2(X,\Z(1)),$$
is precisely the space of weight $2$ Hodge cycles. 
  \end{thm}

Although the  proof is almost immediate, we give a second proof for
surfaces  which although longer  gives more geometric insight.

\begin{proof}[First Proof]
  The proof of proposition \ref{prop:cH2M} actually shows that
$\im c=H^2(X,\Z(1))\cap F^1$.
But this is exactly what we want to prove.
\end{proof}

\begin{proof}[Second Proof]
For this proof, we will assume that $X$ is a surface.
Let $\alpha\in \tilde H^2(X,\Z)$ be the image of  Hodge cycle.
 By proposition \ref{prop:carlson},  
$ \alpha$ corresponds to a divisor $D$ on $\tilde X_0$ whose class in 
$$P(\tilde X_\dt) = Pic^0(\tilde X_0)\backslash Div_h(\tilde
X_1)/R_\partial(\tilde X_1)^*$$
is zero.  In more explicit terms, this means that after translating
$D$ by element of $Pic^0(\tilde X_0)$, we can assume that
$\partial D= (f)$ for some $f\in R_\partial(\tilde X_1)^*$. Recall
that the last condition means that $\partial f=\partial g$, for some
locally constant function $g$. After replacing $f$ by $fg^{-1}$, the relations \eqref{eq:cobound} hold for $(D,f)$.
This gives a class
in $H_M^2(X,\Z(1))$ which maps onto $ \alpha$.
\end{proof}

\section{Cohomological Hodge conjecture for singular varieties}

Extrapolating from theorem \ref{thm:Lef}, suggests the following conjecture:

\begin{con}\label{con:hodge}
Let $X$ be a  complex projective variety be defined over $\overline{\Q}$. Then
every weight $2p$ Hodge cycle in $H^{2p}(X,\Q(p))$
lies in the image of the map from $H_M^{2p}(X,\Q(p))$ constructed in theorem \ref{thm:obstrvan}.
\end{con}

Note that, unlike the usual Hodge conjecture, this is easy to falsify when the condition of being defined over
$\overline{\Q}$ is dropped.  Suppose that $X$ is a union of two smooth components $X_1\cup X_2$
meeting transversally along a smooth variety $Y$. Then we can compute
motivic cohomology using the  resolution
$$ Y\rightrightarrows \tilde X= X_1\coprod X_2\to X$$
by remark \ref{rmk:ncd}.
As soon as we can find an algebraic cycle  $\alpha$ on $\tilde X$ such that $\partial \alpha$
is nonzero  in $CH^*(Y)_\Q$ but zero in the rational Deligne cohomology of $Y$ (or
equivalently both homologically trivial and trivial in the intermediate Jacobian tensor $\Q$),
then we get a counterexample. An explicit example was found by Bloch.

\begin{ex}[Bloch, {\cite[appendix 1]{jannsen}}] Let $S\subset \PP^3$ be a
  smooth surface of degree $\ge 4$ and $Y= Bl_x S$ the blow up of $S$ at a very general
  point $x\in Y$. Then the union $X = Bl_x\PP^3\coprod_{ Y}
  Bl_x\PP^3$ carries a codimension $2$ cycle $\alpha$ as above.
\end{ex}

Bloch and Beilinson \cite[lemma 5.6]{beilinson} have conjectured that when $Y$ is a smooth
projective variety over $\overline{\Q}$, the cycle map from $CH^*(Y)_\Q$ to rational Deligne
cohomology is injective. We can easily see that:

\begin{prop}
  Assuming the usual Hodge conjecture and the Bloch-Beilinson
  conjecture, any  weight $2p$ Hodge cycle in $H^{2p}(X,\Q(p))$
on a projective variety $X$ defined over $\overline{\Q}$ is represented by an algebraic cycle 
$\alpha_0$ on $\tilde X_0$ such that $\partial \alpha_0=0$ in $CH^p(\tilde X_1)_\Q$, for  any
GNPP resolution $\tilde X_\dt$  defined over $\overline{\Q}$.  In
particular, conjecture \ref{con:hodge} holds if  in addition $\tilde X_2=\emptyset$
or more generally  if $\dim \tilde X_2<p-1$
\end{prop}

\begin{proof}
The argument will use the validity of the Hodge conjecture for$\tilde X_0$ and
Bloch-Beilinson for $\tilde X_1$.
A weight $2p$ Hodge cycle on $H^{2p}(X)$ pulls back to an algebraic
cycle $\alpha_0$ on $\tilde X_0$ such that $\partial \alpha_0$ is
homologically trivial.  By proposition \ref{prop:epsilon1}, we can
assume that $\partial \alpha_0 =0$ in $J(\tilde X_1)_\Q$ after modifying
$\alpha_0$ by  adding a
 homologically trivial cycle to it. Thus $\partial \alpha_0=0$ in the
 Chow group. So we can find a higher cycle $\alpha_1$ on $\tilde X_1$
 such that $\delta\alpha_1=\alpha_0$. If $\dim \tilde X_2<p-1$,
 then  trivially $\partial \alpha_1=0$ in Bloch's complex for     $\tilde X_2$. So $\alpha_\dt$
 determines a motivic class.
\end{proof}

While this does not appear to be enough to imply conjecture \ref{con:hodge},
 at the very least, this would appear to rule out ``easy''
counterexamples of the above type over $\overline{\Q}$.

We  state some basic criteria for the conjecture to hold. There are a few
easy cases that can be reduced to the usual Hodge conjecture.

\begin{lemma}
Given a blow up square \eqref{eq:rs}, we have a commutative diagram
with exact rows
  $$
\xymatrix{
& H_M^{2p}(X,\Q(p))\ar[r]\ar^{\alpha}[d] & H_M^{2p}(\tilde X,\Q(p))\oplus H_M^{2p}(Z,\Q(p)) \ar[r]\ar^{\beta}[d] & H_M^{2p}(E,\Q(p))\ar^{\gamma}[d] \\ 
0\ar[r]& \tilde H^{2p}(X,\Q(p))\ar[r] & H^{2p}(\tilde X,\Q(p))\oplus H^{2p}(Z,\Q(p)) \ar[r] & H^{2p}(E,\Q(p))
}
$$
If $\beta$ is surjective and $\gamma$ is injective then $\alpha$ is surjective.
\end{lemma}

\begin{proof}
  The exactness of the top row is  proposition \ref{prop:cdh}, for the
  bottom this is standard, and the commutativity follows from
  functoriality of the cycle map.  The last statement is a consequence
  of the five lemma.
\end{proof}

\begin{cor}
  Suppose that  $X$ has isolated singularities, and possess a resolution $\tilde
  X$ for which the Hodge conjecture holds and such that the
  exceptional divisor $E$ is smooth and has a cellular decomposition
  (in the sense of \cite[ex 19.1.11]{fulton}). Then 
  conjecture~\ref{con:hodge} holds for $X$.
So in particular, the conjecture holds for a cone over a cellular variety.
\end{cor}

\begin{lemma}\label{lemma:funM}
  If $f:X\to Y$ is a map of projective varieties, and $\alpha\in
  H^{2p}(Y,\Q(p))$ lies in the image of $H_M^{2p}(Y,\Q(p))$, then
  $f^*\alpha$ lies in the image of $H_M^{2p}(X,\Q(p))$. In 
  particular, the conclusion holds if $Y$ is smooth and $\alpha$ is algebraic.
\end{lemma}

\begin{proof}
  This follows from naturality of the cycle map.
\end{proof}

\begin{lemma}\label{lemma:prodM}
The image of $H_M^{2*}(X,\Q(*))$ in $H^{2*}(X,\Q(*))$ forms a subalgebra.
\end{lemma}

\begin{proof}
  Let $\alpha = \alpha_1\cup\ldots \cup\alpha_n$, where $\alpha_i\in
  H^{2*}(X,\Q(*))$ is the  image of $\beta_i\in H^{2*}_M(X,\Q(*))$ under the
cycle map.
We can form the product $\beta=\beta_1\cup\ldots\cup\beta_n\in H^{2*}_M(X,\Q(*))$ (section \ref{section:mot}). 
In order to show that $\beta$ maps to $\alpha$, it  is enough to check that their images in $\tilde H^{2*}(X)$ agree
by lemma~\ref{lemma:HCinj}. Let
$\pi:\tilde X\to X$ be a resolution of singularities. Then $\pi^*\beta$
is the usual intersection product $\pi^*\beta_1\cdot \pi^*\beta_2\ldots\in CH^*(\tilde X)_\Q$. This maps
to $\pi^*\alpha_1\cup\ldots\cup\pi^*\alpha_n \in\tilde H^{2*}(X)$.
\end{proof}

\begin{cor}\label{cor:prodM}
  If $\alpha$ is a sum of products of degree $2$ Hodge cycles in $H^2(X,\Q(1))$, then it lies in the image
of $H^{2*}_M(X,\Q(*))$. 
\end{cor}

\begin{lemma}\label{lemma:cdhpi}
  Let $\pi:X\to Y$ be a finite morphism of normal varieties,
then $\pi_*:Sh(X_{cdh})\to Sh(Y_{cdh})$ is exact.
\end{lemma}

\begin{proof}
  By \cite{gk}, the topos $ Sh(Y_{cdh})   $ has enough points, and these correspond to maps of spectra of
Henselian valuation rings to $Y$. Given a point $y:\Spec A\to Y$ in
this sense, let $\{x_i\}$ be the finite set of points of $X$ corresponding
 to the  valuations of $\Gamma(\OO_{\Spec A\times_Y X})$ extending the
 valuation of $A$. It can be checked that
the stalk  $(\pi_*\F)_y =\prod \F_{x_i}$. Therefore given an epimorphism $\F\to \mathcal{G}$ of sheaves,
 $(\pi_*\F)_y \to (\pi_*\mathcal{G})_y$ is surjective for all points $y$. This suffices to prove the lemma.
\end{proof}

\begin{prop}\label{prop:invM}
  Let $X$ be a projective variety with an action by a finite group
  $G$. Let  $\pi:X\to Y= X/G$ be the quotient (which is well known to
   exist in the category of projective varieties). Then $\pi^*$ induces
  an isomorphism
$$H^i_M(Y,\Q(n))\cong H_M^i(X,\Q(n))^G$$
\end{prop}

\begin{proof}
  Given a cdh open $U\subset Y$, and an irreducible  cycle $V\in
  z_{qf}( \A^n\times \Delta^i)(U)_\Q$, then the cycle theoretic pullback
$$\pi^*V = \sum_W \text{card}\{g\in G\mid g|_W=id\} [W] $$
where $W$ runs over irreducible components of $\pi^{-1}U$.
This determines a cycle in $z_{qf}( \A^n\times \Delta^i)(\pi^{-1}U)_\Q$.
Moreover, it is seen to
induce an isomorphism 
$$  z_{qf}( \A^n\times \Delta^\dt)(U)_\Q \cong z_{qf}( \A^n\times \Delta^\dt)(\pi^{-1}U)_\Q^G $$
\cite[ex 1.7.6]{fulton}.
It is  also  compatible with the differential $\delta$, and therefore
it induces an isomorphism
$$H^i(Y_{cdh}, \Q(n)) \cong H^i(Y_{cdh},(\pi_*\Q(n))^G)$$
The functor of $G$-invariants  on $\Q$-modules is well known to be exact, together
with lemma  \ref{lemma:cdhpi},   this implies that
we can write the last group as 
$$ H^i(Y_{cdh}, \pi_*\Q(n))^G=H^i(X_{cdh},\Q(n))^G$$
\end{proof}

\begin{cor}\label{cor:invM}
  If conjecture~\ref{con:hodge} holds for $X$, then it holds for $Y$.
\end{cor}

\begin{proof}
  By assumption we have a surjection $H_M^{2p}(X,\Q(p))\to
  H_H^{2p}(X,\Q(p))$, where the right side denotes the space of weight
  $2p$ Hodge cycles. Therefore we have  surjections
$$
\xymatrix{
 H_M^{2p}(X,\Q(p))^G\ar@{>>}[r] & H_H^{2p}(X,\Q(p))^G \\ 
 H_M^{2p}(Y,\Q(p))\ar@{>>}[r]\ar[u]^{\cong} & H_H^{2p}(Y,\Q(p))\ar[u]^{\cong}
}
$$
\end{proof}

\section{Fibre products of modular surfaces}

As evidence for conjecture \ref{con:hodge},
 we will check it  for the following class of examples. Let
$\Gamma\subseteq SL_2(\Z)$ be a subgroup of finite index such that
$-I\notin \Gamma$. Let $\mathbb{H}$ be the upper half plane and let 
$U=\mathbb{H}/\Gamma$ be the associated modular
curve with smooth projective compactification $C\supset U$. This can
be interpreted as the moduli space of (generalized) elliptic curves
with $\Gamma$-level structures. So in particular we get an associated 
universal family $f:\E\to C$, which is called an elliptic modular
surface \cite{shioda}. This is defined over $\overline{\Q}$.

\begin{thm}\label{thm:ellipticmod}
  Let $f:\E\to C$ be an elliptic modular surface. Then for any $n\ge
  1$, conjecture \ref{con:hodge} holds for  the $n$-fold  fibre product $X=\E\times_C\ldots \times_C\E$.
\end{thm}

Before starting the proof of the theorem, we will need to recall some
facts about elliptic modular surfaces.
Let us assume that   $f:\E\to C$ is a semistable elliptic modular surface. Let $S= C-U$.
 The cohomology $H^2(\E,\Q)$ carries a filtration induced by the Leray
spectral sequence. Since this degenerates \cite[\S 15]{zucker}, we can write
 the subquotients  of $H^2(\E,\Q)$ as
 \begin{eqnarray*}
   L^2 &=&  H^2(C,\Q)\\
L^1/L^2 &=& H^1(C, R^1f_* \Q)\\
L^0/L^1 &=& H^0(C, R^2 f_*\Q)
 \end{eqnarray*}
It is immediate that $L^2$ is generated by the
class of a fibre of $f$.  

The restriction of $R^2f_*\Q$ to $U$ is the
constant sheaf $\Q_U$.  So we have an adjunction map $R^2f_*\Q\to
\Q_C$ leading to an exact sequence
\begin{equation}
  \label{eq:KR2f}
0\to K\to R^2f_*\Q\to \Q\to 0  
\end{equation}
where $K=\bigoplus K_s$ is a sum of sheaves  supported at $s\in S$.  We can interpret this more explicitly
by restricting to a small disk $D$ centered at $s\in S$.  Let $t\in
D-\{s\}$. 
Then the restriction of \eqref{eq:KR2f} to corresponds to the sequence
$$0\to H^0(K_s)\to H^2(X_s,\Q)\stackrel{c}{\to} H^2(X_{t},\Q)=\Q\to 0$$
 The map $c$ is the collapsing map induced by the homotopy equivalence
followed by restriction $X_s\approx f^{-1}X \supset X_t$. 
This leads to a sequence
\begin{equation}
  \label{eq:L0L1}
0\to \bigoplus_s H^0(K_s)\to L^0/L^1\to \Q(-1)\to 0  
\end{equation}
The space on the  right  is generated
by a fundamental class of an irreducible curve which is horizontal
 in the sense having nonzero intersection  number with the general
 fibre.  The space on $H^0(K_s)$   is spanned divisors supported on $X_s$ orthogonal
to the horizontal divisor.
It follows that $L^0/L^1$ is spanned by divisors.

It remains to analyze $L^1/L^2$. Set $\cL=R^1f_*\Q$.
 To begin with,  we claim that 
\begin{equation}
  \label{eq:jjL}
\cL\cong j_*j^*\cL  
\end{equation}
To prove this, it suffices to check isomorphisms at the stalks at each
$s\in S$.   Choose  a small disk $D$ centered at $s$. 
Then we have to show that $H^1(X_s,\Q)\cong H^1(X_t,\Q)^{\pi_1(D^*)}$. 
Semistability implies that fibre $X_s$ is of type $I_N$, i.e. a
polygon of $N$ smooth rational curves, for some $N$. We can choose a symplectic basis $e_1, e_2$ of $H_1(X_t)$ 
such that the $N$ vanishing cycles are all homologous to $e_2$, and
the image of $e_1$ generates $H_1(X_s)$. Thus $H^1(X_s)\to H^1(X_t)$
is injective. The image is precisely the dual $e_1^*$, which by the
Picard-Lefschetz formula spans the invariant cycles. Therefore
\eqref{eq:jjL} holds. Consequently
$$H^1(C,\cL)\cong H^1(C, j_*\cL|_U)$$
The right side can be identified with intersection cohomology $IH^1(C,\cL)$.

 Zucker \cite[thm 7.12]{zucker} showed that intersection cohomology
$IH^1(C,\cL)$ carries an intrinsic Hodge structure
 which is isomorphic $L^1/L^2$. This comes by identifying this with
 $L^2$ cohomology with coefficients in $\cL$.  In a bit more detail,
the local system $\cL_U$ is associated to a polarized variation of Hodge
  structure on $U$ with unipotent local monodromy.  For any such variation,
by work of Schmid
  \cite{schmid} the   vector
  bundle  $\mathcal{V}_U=\cL_U\otimes \OO_U$ with its Hodge
  filtration extends to a filtered bundle $(\mathcal{V},{F})$
  on $C$. The log complex
$$\mathcal{V}\stackrel{\nabla}{\to} \mathcal{V}\otimes \Omega_C^1(\log
S)$$
is filtered by
$${F}^p\to {F}^{p-1}\otimes \Omega_C^1(\log
S)$$
The subcomplex $\mathcal{V}\to \im\nabla$ with induced filtration
 forms part of a cohomological Hodge complex that computes $IH^*(\cL)$.
Returning to  our specific case, we have isomorphisms
$$V^1={F}^1\cong f_*\Omega^1_{X/C}(\log f^{-1}S)=
f_*\omega_{X/C}$$
 and
$$V^0={F}^0/{F}^1\cong R^1f_*\OO_X$$ 
by  \cite{steenbrink}. 
An easy  computation
shows that
$$
IH^1(\cL)_\C^{(p,q)} =
\begin{cases}
H^1(C,V^0) & (p,q)=(0,2)\\
  H^1(C, V^1\stackrel{\kappa}{\to} V^0\otimes \Omega_C^1(\log S))& (p,q)=(1,1)\\
H^0(C,V^1\otimes \Omega^1_C(\log S)) &(p,q)=(2,0)
\end{cases}
$$
where $\kappa$ is the Kodaira-Spencer class.
We can repackage this  by
  defining  the graded vector bundle  $V= V^0 \oplus V^1$ with Higgs field
$$\theta=
\begin{pmatrix}
  0 &0\\ \kappa & 0
\end{pmatrix}
:V\to V\otimes \Omega^1_C(\log S)
$$
Then $IH^1(\cL)$ is the first hypercohomology of the last complex, and
the $(p,q)$ decomposition can be recovered from the induced grading. This viewpoint
is more convenient for 
analyzing $IH^1(\cL^{\otimes n})$ below.
In order to do this, observe that given two locally unipotent polarized variations of Hodge
structure with associated graded Higgs bundles
$(W,\theta)$ and $(W',\theta')$, the tensor
 product is associated to $W''=W\otimes W'$ with Higgs field $\theta\otimes 1+1\otimes
\theta'$ and grading
$$(W'')^i=\bigoplus_{j+k=i} W^i\otimes (W')^j$$

  Shioda \cite[eq (4.12)]{shioda} showed  that $\dim
IH=2p_g(\E)$. Together with the fact that
 $p_g(\E)= \dim H^1(C, R^1f_*\OO_X) = \dim IH^{(0,2)}$, we can
 conclude that 
 \begin{equation}
   \label{eq:IH11}
\dim IH^{(1,1)}=0   
 \end{equation}
Since $\kappa$ is nonzero, we conclude that
$$ V^1\stackrel{\kappa}{\to} V^0\otimes \Omega_C^1(\log S)\cong
\coker \kappa [-1]$$
in the derived category.
Combing this with \eqref{eq:IH11} implies that $\kappa$ is isomorphism.

To extend this analysis to nonsemistable surfaces, we observe the
following.

\begin{lemma}\label{lemma:semi}
  If $\E\to C$ is an elliptic modular surface, then there exists a
  Galois cover $p:C'\to C$ such that $\E'=\E\times_{C} C'$ is birational to a
  semistable modular surface.
\end{lemma}

\begin{proof}
  Let $\Gamma\subset SL_2(\Z)$ be the group associated to $\E$.
Then we may take $C'\to C$ to be the modular curve associated to
$\Gamma\cap \Gamma(N)$, where $\Gamma(N)$ is the principal congruence
subgroup of level $N\ge 3$. It is known that all singular fibres of the
elliptic modular surface corresponding to $\Gamma(N)$ will be of type $I_N$ 
\cite[ex 5.4]{shioda}, and consequently semistable. Semistability will
persist over $C'$.
\end{proof}

With the notation as in the lemma, let $\pi:\tilde \E\to \E'$ be the minimal
resolution, $f':\E'\to C$ and $\tilde f:\tilde \E\to C$ the
projections, and
let $G$ be the Galois group of $C'/C$. We claim that the above results
carry over to $f':\E\to C'$. More specifically, there are isomorphisms
or exact sequences

\begin{equation}
  \label{eq:Rf'}
  \begin{split}
&f'_*\Q = p_*\Q\\
& R^1f'_*\Q = j_*j^* \cL,\quad \cL'=R^1f'_*\Q\\
&0\to \bigoplus_{s\in S} K_s'\to R^2f'_*\Q\to p_*\Q\to 0
\end{split}
\end{equation}
where $K_s'$ is supported on $S$ and spanned by algebraic cycles
supported on $\E'_s$. We also have that
$$IH^1(C,\cL')^{11} =H^1( (V')^1\stackrel{\kappa}{\to} (V')^0\otimes \Omega_C^1(\log S) )=0$$
where the graded Higgs bundle is defined as above. These
statements follow from straightforward modifications of
the  previous arguments. We also have that:

\begin{lemma}\label{lemma:leray}
  The Leray spectral sequence for $f'$ degenerates at $E_2$
\end{lemma}

\begin{proof}
The only differential
that could be nonzero is indicated as $d'_2$ below.
$$
\xymatrix{
 H^0(R^1f'_*\Q)\ar[r]^{d'_2}\ar[d]^{\cong} & H^2(f'_*\Q)\ar[d]^{\cong} \\ 
 H^0(R^1\tilde f_*\Q)\ar[r]^{d''_2} & H^2(\tilde f_*\Q)
}
$$
The vertical maps are easily seen to be isomorphisms by \eqref{eq:Rf'}
and the analogous facts for $\tilde f$. Since $d''_2=0$ by
\cite[cor 15.15]{zucker}, we can conclude that $d'_2=0$. 
\end{proof}

\begin{lemma}\label{lemma:HL}
  Given a polarized variation of Hodge structure $\cL$ on $U$, cup
  product with the fundamental class $[C]$ induces an isomorphism
$$ IH^0(C,\cL)\cong IH^2(C,\cL )$$
\end{lemma}

\begin{proof}
This is a special case of the hard Lefschetz  theorem of  Saito \cite[thm
5.3.1]{saito}, but this can proved more directly as follows.  Both
groups are represented by spaces of  $L^2$ $\cL$-valued harmonic forms \cite[\S 7]{zucker}. 
Using this representation and the K\"ahler identities \cite[\S 2]{zucker}, the map, which given by wedging with the
K\"ahler form, is seen to be  an isomorphism by the usual argument
\cite[pp 118-122]{gh}.
\end{proof}

\begin{proof}[Proof of theorem \ref{thm:ellipticmod}]
By corollary \ref{cor:invM} and lemma \ref{lemma:semi}, it is enough
to prove the conjecture for
 $X'=\E'\times_C\ldots \times_C\E'$, because $X=X'/G^n$.
From this point onwards, there is no need to refer to the original
variety $X$. So in the interest of simplifying the notation, we omit
the primes and write $X,
f,\cL,\ldots$ instead of $X', f', \cL',\ldots$
By corollary \ref{cor:prodM}, it is enough to prove that the algebra
of Hodge cycles in $H^{2*}(X)$ is generated by degree $2$ cycles. If we
denote the space of Hodge cycles (respectively products of degree $2$
Hodge cycles) by $H^{2*}_{Hodge}(X)$ (respectively $H^{2*}_{Hodge,2}(X)$),
then  we have to show that $\dim H^{2*}_{Hodge}(X)= \dim H^{2*}_{Hodge, 2}(X)$.
Toward this end, it suffices to prove that $\dim H^{2*}_{Hodge}(X)\cap H_i=\dim
H^{2*}_{Hodge, 2}\cap H_i$ for any possibly noncanonical decomposition
$H^*(X)=\bigoplus H_i$.

The Leray spectral sequence
$$H^p(C, R^qF_*\Q) \Rightarrow H^{p+q}(X,\Q)$$
is  compatible with mixed Hodge structures
\cite{arapura}.  This degenerates by an argument similar to the proof
of lemma~\ref{lemma:leray}.
 Thus we will have a
noncanonical decomposition
\begin{equation}
  \label{eq:Leray}
 \tilde H^{2k}(X,\Q) \cong \bigoplus_{i+j=2k} H^i(C, R^jF_*\Q)  
\end{equation}
We will show that Hodge cycles on the spaces on the right are products
of degree $2$ cycles. 
We can decompose the direct
images 
$$R^j F_*\Q = \bigoplus_{{a+b+c=n}\atop{b+2c=j}}  (\underbrace{(f_*\Q)^{\otimes a}\otimes
(R^1f_*\Q)^{\otimes b}\otimes (R^2f_*\Q)^{\otimes c}}_{R(a,b,c)})^{N(a,b,c)}$$
using K\"unneth's formula, where $N(a,b,c)$ is some exponent whose
precise value  is unimportant for us. The sequence of \eqref{eq:Rf'} gives an exact sequence
$$0\to \bigoplus_{s\in S} K_s(c) \to (R^2f_*\Q)^{\otimes c}\to
\Q^{\otimes c}\cong \Q\to 0$$
where $K_s(c)$  is not the Tate twist, it is merely a notation for a
certain sky scraper sheaf supported at $s$.  It
decomposes noncanonically as
$$K_s(c) \cong \bigoplus_{a+b=c} (K_s^{\otimes a}\otimes \Q^{\otimes
  b})^{N(a,b)}\cong \bigoplus_{a+b=c} (K_s^{\otimes a})^{N(a,b)}$$
The components fit into  exact sequences
$$
\xymatrix{
 0\ar[r] &  \bigoplus_s \underbrace{\Q_s^{\otimes a}\otimes
   (j_*j^*\cL^{\otimes b})_s\otimes K_s(c)}_{}\ar[r]\ar@{}[d]^{\cong}
 &  R(a,b,c)\ar[r] & {}\underbrace{\Q^{\otimes a}\otimes  j_*j^*\cL^{\otimes  b}\otimes \Q^{\otimes c}}_{} \ar[r]\ar@{}[d]^{\cong} & 0 \\
  & K_s(c)&  &j_*j^*\cL^{\otimes  b} & 
}
$$
Thus we have (noncanonical) isomorphisms
\begin{equation}
  \label{eq:Rabc0}
H^0(C,R(a,b,c))\cong  \left(\bigoplus_s H^0(K_s(c))\right)\oplus  (\cL^{\otimes   b})^{\pi_1(U)}
\end{equation}

\begin{equation}
  \label{eq:Rabc1}
H^i(C,R(a,b,c))\cong IH^i(C,\cL^{\otimes  b}),\, i\ge 1  
\end{equation}
We analyze each of these summands in turn, and show that Hodge cycles
in them are spanned by degree $2$ Hodge cycles.

\begin{enumerate}

\item 
The Zariski closure of the  image of $\pi_1(U)$ under the monodromy
representation associated to $\cL$ is $SL_2(\Q)$. So by classical
invariant theory \cite[appendix F]{fh}, $ (\cL^{\otimes
  b})^{\pi_1(U)}$ is a sum of  products of sections of $ (\cL^{\otimes
  2})^{\pi_1(U)}$, and therefore a sum of products of
degree $2$ Hodge cycles. 

\item  The spaces $H^0(K_s(c))$
can be further decomposed into of sums of tensor powers of
$H^0(K_s)$, and  each of these spaces is generated by degree $2$
classes.

\item Next, we turn to $IH^2(C,(\cL^{\otimes b}))$. By lemma
  \ref{lemma:HL} there is an isomorphism
$$ (\cL^{\otimes  b})^{\pi_1(U)}= IH^0(C,(\cL^{\otimes b})
)\stackrel{\sim}{\longrightarrow}IH^2(C,(\cL^{\otimes b}) )$$
given by cupping with the fundamental class $[C]$.
With this isomorphism, we
see that these groups are generated by degree $2$ Hodge cycles.

\item Finally consider,
$$T:=
IH^1(C,\cL^{\otimes(2q-1)})^{(q,q)}$$
By  previous remarks, $T$
 can be computed as the $q$th summand of  the first hypercohomology of the 
graded Higgs bundle $(V,\theta)^{\otimes (2q-1)}$. In more explicit terms, $T$ is the $1$st
hypercohomology of the complex
\begin{equation}
  \label{eq:kunneth}
  \bigoplus_{\sum i_k=q} V^{i_1}\otimes\ldots\otimes V^{i_{2q-1}}\to \bigoplus_{\sum j_k=q-1} 
V^{j_1}\otimes\ldots \otimes V^{j_{2q-1}}\otimes \Omega_C^1(\log S)
\end{equation}
The differential  is given as a sum of maps $1\otimes \kappa\otimes 1$.
This is acyclic because $\kappa$ is an isomorphism. Therefore $T=0$.

\end{enumerate}

\end{proof}

When $\E\to C$ is a semistable elliptic modular surface, the
singularities of $X= \E\times_c\ldots\times_c\E$ are
toroidal. Therefore we have a toroidal resolution of singularities
$\pi:\tilde X\to X$ (cf \cite{gordon}).

\begin{cor}[Gordon]
  The Hodge conjecture holds for $\tilde X$.
\end{cor}

\begin{proof}
  In outline, the cohomology of $\tilde X$ is generated by the image
  of $H^*(X)$ and algebraic cycles supported on the exceptional locus
  of $\pi$. The Hodge cycles in $\pi^*H^*(X)$ lie in the image of
  $H^{2*}_M(X,\Q(*))$, which factors through $CH^*(X)_\Q$.
\end{proof}

Gordon's proof  is somewhat different.
As noted earlier, there does not seem to be anyway of going backwards
and deducing the theorem from this result. 



\begin{thebibliography}{1234}

\bibitem[AJS]{ajs} L. Alonso Tarrio, A.  Jeremias Lopez, M. Souto
  Salorio, {\em Localization in categories of complexes and unbounded
    resolutions.} Canad. J. Math. 52 (2000)

\bibitem[A]{arapura} D. Arapura, {\em The Leray spectral sequence is motivic}, 160 (2005), no. 3, 567-589. 

\bibitem[ACK]{ack} D. Arapura, X. Chen, S-J. Kang, {\em The smooth
    center of the cohomology of a singular variety. }
Hodge theory, complex geometry, and representation theory, 1–19, Contemp. Math., 608, Amer. Math. Soc.,(2014)

\bibitem[BS]{bs} L. Barbieri-Viale, V. Srinivas, {\em The Neron-Severi
  group and the mixed Hodge structure on $H^2$}, J. Reine
Angew. Math. 450 (1994), 37-42

\bibitem[Be]{beilinson} A. Beilinson, {\em Height pairings between
    algebraic cycles}, K -theory, arithmetic and geometry (Moscow, 1984–1986), 1–25, Lecture Notes in Math., 1289, Springer, Berlin, 1987.

\bibitem[B]{bloch} S. Bloch, {\em  Algebraic cycles and higher
    K-theory},  Adv. in Math. 61 (1986), no. 3, 267-304. 

\bibitem[BFFU]{bff} S.  Boucksom, T. de Fernex,  C. Favre, S. Urbinati, {\em Valuation spaces and multiplier 
ideals on singular varieties},  Recent advances in algebraic geometry, 29-51, London Math. Soc. Lecture Note Ser., 417, Cambridge Univ. Press, Cambridge, 2015

\bibitem[C]{carlson} J. Carlson, {\em Extensions of mixed Hodge
    structures}, Journ\'ees de G\'eometrie Alg\'ebrique d'Angers,
  1979, pp. 107–127, Sijthoff \& Noordhof  (1980)

\bibitem[C2]{carlson2} J. Carlson, {\em The one-motif of an algebraic
    surface}, Compositio Math. 56 (1985), no. 3, 271-314.

\bibitem[D1]{deligneH} P. Deligne, {\em Th\'eorie de Hodge II, III},
  Inst. Hautes \'Etudes Sci. Publ. Math. No. 40, 44 (1971, 1974)

\bibitem[D2]{deligneW} P. Deligne, {\em La conecture de Weil I},
  Inst. Hautes \'Etudes Sci. Publ. Math. No. 43 (1974), 273–307.

\bibitem[dB]{dubois} P. du Bois, {\em Complexe de de Rham filtr\'e
      d'une vari\'et\'e singuli\`ere},  Bull. Soc. Math. France 109 (1981), no. 1, 41–81. 


\bibitem[FV]{fv} E. Friedlander, V. Voevodsky, {\em Bivariant cycle
    cohomology}, in  Annals of Mathematics Studies, 143. Princeton University Press,  2000. 

 

\bibitem[F1]{fulton1} W. Fulton, {\em Rational equivalence on singular
  varieties}, Inst. Hautes \'Etudes Sci. Publ. Math. No. 45 (1975), 147–167. 

\bibitem[F]{fulton} W. Fulton, {\em Intersection theory}, Second
  edition. Springer-Verlag, Berlin, (1998)

\bibitem[FH]{fh} W. Fulton, J. Harris, {\em Representation theory. A
    first course.} GTM 129. Springer-Verlag, New York, (1991)

\bibitem[GK]{gk} O. Gabber, S. Kelly, {\em Points in algebraic
    geometry},  J. Pure Appl. Algebra 219 (2015), no. 10, 4667-4680.

\bibitem[G]{gordon} B. Gordon, {\em Algebraic cycles and the Hodge
    structure of Kuga  fiber variety }, Trans. Amer. Math. Soc. 336 (1993), no. 2, 933-947

\bibitem[GH]{gh} P. Griffiths, J. Harris, {\em Principles of algebraic
  geometry}, Wiley-Interscience (1978)

\bibitem[GNPP]{gnpp} F. Guillen, V. Navarro Aznar, P. Pascual Gainza,
  F. Puerta, {\em Hyperr\'esolutions cubiques et descent
    cohomologique}, LNM 1335, Springer-Verlag (1988)

\bibitem[Ha]{hanamura} M. Hanamura, {\em Homological and cohomological
    motives of algebraic varieties}, Invent. Math. (2000)

\bibitem[H]{hormander} L. H\"ormander, {\em The analysis of linear
    partial differential operators I}, Second edition,
 Springer-Verlag, (1990)

\bibitem[J]{jannsen} U. Jannsen, {\em Mixed motives and algebraic K-theory}
 Lecture Notes in Mathematics, 1400. Springer-Verlag, Berlin, 1990

\bibitem[dJ]{dJ} J. de Jong, {\em Smoothness, semi-stability and alterations}
Inst. Hautes \'Etudes Sci. Publ. Math. No. 83 (1996), 51–93.  

\bibitem[KLM]{klm} M. Kerr, J. Lewis, S. M\"uller-Stach, {\em The
    Abel-Jacobi map for higher Chow groups, } Composito 142 (2006)

\bibitem[K]{kimura} S. Kimura, {\em Fractional intersection and
    bivariant theory},  Comm. Algebra 20 (1992), no. 1, 285-302.

\bibitem[L]{laumon} G. Laumon, {\em  Homologie \'etale.} S\'eminaire de
    g\'eom\'etrie analytique (\'Ecole Norm. Sup., Paris, 1974-75),
 pp. 163–188. Asterisque, No. 36-37, Soc. Math. France, Paris, 1976. 


 
\bibitem[Le]{levine} M. Levine, {\em Bloch's higher Chow groups
    revisited },
K-theory (Strasbourg, 1992). Ast\'erisque No. 226 (1994), 10, 235–320.


\bibitem[MVW]{mvw} C. Mazza, V. Voevodsky, C. Weibel, {\em Lecture notes on motivic cohomology.}
 Clay Mathematics Monographs, 2. American Mathematical Society, (2006)

 \bibitem[M]{mumford} D. Mumford, {\em The topology of normal
     singularities of an algebraic surface}, Inst. Hautes Études Sci. Publ. Math. No. 9 1961 5-22. 

 \bibitem[N]{nart} E. Nart, {\em The Bloch complex in codimension one
     and arithmetic duality},  J. Number theory 32 (1989)
 
 \bibitem[S]{s} B. Saint-Donat, {\em  Techniques de descent
     cohomologique}, SGA4, Springer Lect Notes 270,  (1972)

 \bibitem[Sa]{saito} M. Saito, {\em Modules de Hodge Polarazibles},
   Publ. Res. Inst. Math. Sci. 24 (1988), no. 6, 849-995 (1989). 

\bibitem[Sc]{schmid} W. Schmid, {\em Variation of Hodge structure},
  Invent. Math. 22 (1973), 211-319. 


\bibitem[Sh]{shioda} T. Shioda, {\em On elliptic modular surfaces}. J. Math. Soc. Japan 24 (1972), 

 \bibitem[Sp]{spaltenstein}  N. Spaltenstein, {\em  Resolutions of unbounded complexes.} Compositio Math. 65 (1988)

 \bibitem[St]{steenbrink} J. Steenbrink, {\em Limits of Hodge structures}, Invent. Math. 31 (1975/76), no. 3, 229–257. 

 \bibitem[SV]{sv} A. Suslin, V. Voevodsky, {\em Relative cycles and
     Chow sheaves}, in Annals of Mathematics Studies, 143. Princeton University Press,  2000. 

 \bibitem[T]{totaro} B. Totaro, {\em Chow groups, Chow cohomology, and
     linear varieties},  Forum Math. Sigma 2 (2014)


 \bibitem[W]{weibel} C. Weibel, {\em Products in higher Chow groups and motivic cohomology}, Algebraic K-theory
(Seattle WA 1997) 305-315, Proc Sympos. Pure Math 67, AMS (1999).

\bibitem[Z]{zucker} S. Zucker, {\em Hodge theory  with degenerating
    coefficients}, Ann. of Math. (2) 109 (1979), no. 3, 415-476. 

\end{thebibliography}
\end{document}